\documentclass[aos]{imsart}

\RequirePackage{amsthm,amsmath,amsfonts,amssymb,mathtools,bm}
\RequirePackage[authoryear,round]{natbib}
\RequirePackage[colorlinks,citecolor=blue,linkcolor=blue,urlcolor=blue]{hyperref}
\RequirePackage{microtype}

\startlocaldefs
\numberwithin{equation}{section}
\theoremstyle{plain}
\newtheorem{theorem}{Theorem}[section]
\newtheorem{proposition}[theorem]{Proposition}
\newtheorem{corollary}[theorem]{Corollary}
\newtheorem{lemma}[theorem]{Lemma}
\theoremstyle{definition}
\newtheorem{assumption}[theorem]{Assumption}

\newtheorem{example}[theorem]{Example}
\newtheorem{remark}[theorem]{Remark}

\newcommand{\E}{\mathbb E}
\newcommand{\Pp}{\mathbb P}
\newcommand{\R}{\mathbb R}
\newcommand{\cA}{\mathcal A}
\newcommand{\cE}{\mathcal E}
\newcommand{\cF}{\mathcal F}
\newcommand{\cG}{\mathcal G}
\newcommand{\cK}{\mathcal K}
\newcommand{\cM}{\mathcal M}
\newcommand{\cP}{\mathcal P}
\newcommand{\cQ}{\mathcal Q}
\newcommand{\cR}{\mathcal R}

\newcommand{\cY}{\mathcal Y}
\newcommand{\dd}{\mathop{}\!\mathrm d}
\newcommand{\Var}{\operatorname{Var}}
\newcommand{\Cov}{\operatorname{Cov}}
\newcommand{\osc}{\operatorname{osc}}

\newcommand{\ind}{\mathbf 1}
\newcommand{\eps}{\varepsilon}

\endlocaldefs

\begin{document}

\begin{frontmatter}
\pdfsubject{Preprint manuscript}
\title{Pre-Disclosure Experiment Menus: Oracle-Relative Risk and Joint Sample--Menu Asymptotics}
\runtitle{Pre-Disclosure Experiment Menus}

\begin{aug}
\author[A]{\fnms{Xinyu}~\snm{Song}\ead[label=e1]{song.xinyu@mail.shufe.edu.cn}}
\address[A]{School of Statistics and Data Science, Shanghai University of Finance and Economics\printead[presep={,\ }]{e1}}
\end{aug}

\begin{abstract}
We study a resolution problem in local asymptotic decision theory: individual
risks may admit Gaussian approximations that do not determine their vanishing
difference.  A finite menu of experiments is installed before context
disclosure, although observations may be routed adaptively afterward.  A
greatest-element Blackwell order collapses adaptive routing to the best
installed experiment and reduces the fixed-menu excess to an
inverse-information distortion with frontier \(A_k\).  We develop
differentiated, all-prior posterior transfer along a one-dimensional
degradation chain and establish
\[
 F_{n,k_n}(H_n)=A_{k_n}\{1+o(1)\}
\]
for every diverging menu sequence with positive frontier and every admissible
localization radius, without an additional direct sample-menu restriction.
The transfer is exact under Gaussian degradation.  Prior-free
likelihood-generator conditions imply it for jump generators and are verified
for binary attenuation, Poisson thinning, and negative-binomial thinning.
If the distortion is uniformly quadratic on an Ahlfors-regular oracle image
of dimension \(r\), then \(A_k\asymp k^{-2/r}\), and the original-scale excess
mean squared error is of order \(n^{-1}k^{-2/r}\).  Calibrated Poisson sensor
and radial-qubit measurement menus illustrate the result.  A triangular
Gaussian counterexample shows why pointwise Gaussian convergence is
insufficient.
\end{abstract}


\end{frontmatter}

\section{Introduction}
\label{sec:introduction}

We study a resolution problem in local asymptotic decision theory.  Ordinary
local asymptotic approximations describe each rescaled risk at order one, but
they do not in general determine a difference between two risks that tends to
zero.  This issue arises in experiment design when the available experiments
must be selected before the statistical task is revealed.  A task-specific
oracle may choose its experiment after disclosure, whereas a shared menu of at
most \(k\) experiments cannot.  For fixed \(k\), local asymptotic minimax
theory identifies the resulting oracle gap.  When \(k=k_n\to\infty\), however,
that gap vanishes and may fall below the resolution of an ordinary LAN
approximation.  The purpose of this paper is to recover the risk difference at
its own shrinking scale.

More precisely, a menu \(V\subset\cA\), with \(1\leq |V|\leq k\), is chosen
before a context \(q\in\cQ\) is disclosed.  Conditional on \(q\), a policy may
route each of \(n\) observations among the experiments in \(V\), using the
previous actions and outcomes, but it cannot add a new experiment.  The
benchmark may choose the best action in the full space \(\cA\) after observing
\(q\).  We compare the worst-context minimax risks of the installed menu and
this context-specific oracle.  The sample size and the menu size have different
roles: \(n\) controls estimation under a selected experiment, whereas \(k\)
controls how closely the installed menu approximates the oracle experiment.
The formulation covers calibrated sensor libraries, quantum measurement
settings, and diagnostic protocols that must be approved before the target
population is identified.

The timing is related to finite adaptability in robust optimization
\citep{BertsimasCaramanis2010,SubramanyamEtAl2020}.  Geometric properties and
performance guarantees for finitely adaptable policies are studied by
\citet{BertsimasGoyalSun2011}.  More recent work separates several aspects of
the approximation: \citet{Kurtz2026} bounds the number of policies required for
specified guarantees, \citet{KedadSidhoumMedvedevMeunier2023} establishes
asymptotic optimality as the number of policies increases under a corrected
continuity condition, and \citet{RezaeiWeiHan2026} studies convergence of
polyhedral policy classes.  These problems optimize robust recourse decisions.
In the present setting, the installed actions generate probability laws, and
their relative value is determined by a repeated-sampling decision problem.

Related statistical-design problems impose different restrictions.
Standardized maximin and robust optimal design compare one common design with
parameter-specific optima \citep{BraessDette2007}, while a menu permits
post-disclosure routing to its best member.  Efficient allocation and optimal
experimentation with the full action set available are studied by
\citet{Armstrong2022}, \citet{Adusumilli2025}, and \citet{LiZhao2025}.
Sequential and adaptive limit experiments are developed by
\citet{HiranoPorter2023} and \citet{Adusumilli2026}.  These procedures choose
among actions during experimentation; they do not restrict the action set by a
finite catalog installed before the context is observed.

The high-resolution geometry is connected with quantization.  The exponent
\(2/r\) agrees with the standard squared-distortion quantization rate in
\(r\) dimensions \citep{GrafLuschgy2000}; scalar quantization for estimation is
considered by \citet{FariasBrossier2014}.  In information economics,
\citet{BergemannBonattiSmolin2018} and \citet{BergemannYehZhang2021} study the
design and pricing of finite information.  Separately, post-measurement
information in quantum decision problems is studied by
\citet{BallesterEtAl2008}, \citet{GopalWehner2010}, and
\citet{CarmeliEtAl2018}.  In a strategic testing problem,
\citet{Hancart2026} lets a privately informed agent select from an offered menu,
so the choice itself conveys information.  Here the context is observed by the
platform, routing is nonstrategic, and the objective is oracle-relative local
minimax risk under repeated sampling.

The first step is an exact reduction.  If the one-observation experiments in
a menu have a greatest element in Blackwell order after the context is fixed,
then repeated use of that element dominates the full record generated by any
randomized and history-adaptive menu policy.  This is a pathwise consequence
of comparison of experiments \citep{Blackwell1953,DeGroot1962}, and it holds
for every sample size, parameter space, and loss.  It fails when the installed
experiments are incomparable, because mixed allocations may then contain
information unavailable from repetition of any one action.

The reduction identifies the fixed-menu excess risk at an interior parameter
value.  Let \(I(q,a)\) denote one-observation Fisher information and let
\(b(q)\) be the derivative of the context-dependent target.  If
\(a_\star(q)\) is the oracle action, the relevant distortion is
\[
 D(q,a)=b(q)^2\{I(q,a)^{-1}-I(q,a_\star(q))^{-1}\},
\]
with the zero-information and locally flat cases defined separately below.
The optimized frontier is
\[
 A_k=\inf_{1\leq |V|\leq k}\sup_{q\in\cQ}\min_{a\in V}D(q,a).
\]
After statistically equivalent actions are identified through a common task
quotient, the Hessian of \(D\) at the oracle determines the local menu
geometry.  A uniformly quadratic distortion on an Ahlfors-regular oracle
image of dimension \(r\) gives \(A_k\asymp k^{-2/r}\), in agreement with the
standard squared-distortion quantization exponent.  On the original target
scale, the corresponding excess mean squared error is of order
\(n^{-1}k^{-2/r}\).  This fixed-menu calculation does not by itself determine
the finite-sample excess when \(A_{k_n}\to0\).

To obtain a relative approximation at the shrinking scale, we parametrize an
ordered family by inverse information \(v\), with larger \(v\) representing a
Blackwell degradation, and differentiate Bayes risk along this chain.  A
nearly least-favorable prior in the Gaussian experiment must have posterior
variance close to the noise variance.  A uniform comparison over all priors
transfers this property to the original experiment, after which priors from
the two endpoints bound the minimax increment.  Under the stated uniformity
and smooth-target conditions, the resulting joint law is
\[
 F_{n,k_n}(H_n)=A_{k_n}\{1+o(1)\}.
\]
This holds for every diverging menu sequence with positive frontier and an
admissible localization radius; the theorem imposes no additional direct
relation between \(n\) and \(k_n\).  The radius may depend on the frontier, so
the statement is not a claim for every prespecified diverging neighborhood.

The all-prior comparison is derived from prior-free conditions on the sampling
laws.  They consist of a multiplicative local Gaussian likelihood comparison,
quadrature and tail bounds, and edge-score calibration for a parameter-free
degradation generator.  Prior mixing preserves the likelihood comparison, and
the calibrated generator converts the posterior Dirichlet form into its
Gaussian counterpart.  Related semigroup energy identities have been studied
for predictive Kullback--Leibler regret by \citet{TakanashiMcAlinn2026}; here
they are used to compare squared-error Bayes-risk derivatives and a shrinking
minimax increment.  A triangular Gaussian counterexample shows that pointwise
Gaussian convergence, even at every fixed degradation level, is insufficient
for this purpose.

The transfer is exact for Gaussian degradation.  We verify the primitive
conditions for binary attenuation, Poisson thinning, and negative-binomial
thinning; the last model provides an overdispersed count experiment
\citep{ZhuJoe2010}.  Two applications make the menu calculation explicit.  A
calibrated Poisson sensor catalog gives an exact squared-distance frontier on
a general compact action space.  Radial qubit measurements give a
non-Euclidean binary example in which the action space is projective.  For the
equatorial qubit model we also prove a separate expansion at the prespecified
radius \(H_n=n^{1/6}\), with an explicit absolute remainder.

The assumptions serve distinct purposes.  The greatest-element order reduces
adaptive routing, the likelihood-generator conditions resolve the shrinking
risk difference, and the quadratic/Ahlfors conditions determine the rate of
the fixed-menu frontier.  We do not treat incomparable vector experiments,
for which mixed allocations may be necessary, or claim that ordinary DQM or
LAN implies differentiated all-prior transfer.  Sharp covering constants are
derived only in the full-Bloch qubit specialization.

The rest of the paper is organized as follows.  Section~\ref{sec:contract}
defines the menu experiment and proves the adaptive collapse.
Section~\ref{sec:fixed-menu} derives the fixed-menu local frontier and its
high-resolution geometry.  Section~\ref{sec:differential} develops the
differential minimax transfer and the joint sample-menu theorem.
Section~\ref{sec:models} verifies the assumptions in four ordered models and
gives the Poisson sensor and radial-qubit applications.
Section~\ref{sec:discussion}
discusses the statistical boundary of the result.  Technical posterior
comparisons and the primitive transfer proof are given in Sections S3--S6 of
the Supplementary Material.

\section{The pre-disclosure menu experiment}
\label{sec:contract}

Let \(\Theta\subseteq\R\) be a parameter space, let \(\cQ\) be a context
space, and let \(\cA\) be a space of experiment actions.  For each
\((q,a)\in\cQ\times\cA\), one observation is drawn from the statistical
experiment
\[
 \cE_{q,a}=(\cY_{q,a},\cF_{q,a},
             \{P_{\theta,q,a}:\theta\in\Theta\}).
\]
For a fixed nonempty finite menu \(V\), place the labeled observations in the
measurable disjoint union
\[
 \mathsf Y_{q,V}=\bigsqcup_{a\in V}(\{a\}\times\cY_{q,a}),
 \qquad
 \mathsf F_{q,V}
 =\left\{\bigsqcup_{a\in V}(\{a\}\times B_a):
          B_a\in\cF_{q,a}\right\}.
\]
The complete \(n\)-stage record is an element of
\((\mathsf Y_{q,V}^n,\mathsf F_{q,V}^{\otimes n})\), and a stage-\(i\) policy
is a Markov kernel from the preceding record to the finite set \(V\).
A nonempty finite menu \(V\subset\cA\) is installed before \(q\) is
disclosed.  At stage \(i\), after observing \(q\) and the recorded history
\((A_j,Y_j)_{j<i}\), a policy draws \(A_i\in V\) from a parameter-free
kernel and then observes \(Y_i\sim P_{\theta,q,A_i}\).  External
randomization is allowed, and the action labels are retained in the record.
The menu is shared across contexts, but the policy and terminal decision may
depend on the disclosed context.

For two experiments \(\cE\) and \(\cG\) with the same parameter space, write
\(\cE\succeq_B\cG\) if \(\cG\) is obtained from \(\cE\) through a Markov
kernel that does not depend on \(\theta\).  This is the Blackwell order.

\begin{theorem}[Adaptive collapse]
\label{thm:collapse}
Fix \(q\), a finite menu \(V\), and an action \(a_V\in V\).  The following
statements are equivalent:
\begin{enumerate}
\item \(\cE_{q,a_V}\succeq_B\cE_{q,a}\) for every \(a\in V\);
\item for every \(n\geq1\), \(\cE_{q,a_V}^{\otimes n}\) Blackwell-dominates
the complete action-outcome record generated by every randomized and
history-adaptive policy using \(V\);
\item for \(n=1\), \(\cE_{q,a_V}\) Blackwell-dominates the record generated
by every deterministic policy using one action in \(V\).
\end{enumerate}
Consequently, all menu policies reduce to repeated use of one action for every
decision problem if and only if the one-observation menu has a greatest
element.
\end{theorem}

\begin{proof}
Suppose the first statement holds, and choose a parameter-free garbling
kernel \(K_{a\leftarrow a_V}\) for every \(a\in V\).  Starting from
independent observations \(Z_1,\ldots,Z_n\) from
\(\cE_{q,a_V}\), run the policy recursively.  At stage \(i\), draw \(A_i\)
from its policy kernel using the previously simulated record, then draw
\[
 Y_i\mid(A_i,Z_i)\sim K_{A_i\leftarrow a_V}(\,\cdot\mid Z_i).
\]
This recursion is a Markov kernel from \(Z_{1:n}\) to the complete record,
independent of \(\theta\).  Induction over \(i\) shows that the simulated
record has the policy law under every \(\theta\), proving the second
statement.  The second statement implies the third by restriction.  For the
converse, apply the third statement to the deterministic policy that selects
an arbitrary \(a\in V\), and delete its deterministic action label.
\end{proof}

Theorem~\ref{thm:collapse} is an exact comparison of experiments.  In
particular, it does not rely on Fisher information, unbiasedness, or an
asymptotic approximation.  Its decision-theoretic consequence is classical;
see \citet{DeGroot1962}.  The pathwise construction is included to make clear
that retaining the action labels and allowing history dependence do not
alter the conclusion.

The greatest-element assumption holds in the three chains used below.
For binary attenuation,
\[
 P_\theta^c(Y=y)=\frac{1+yc\theta}{2},\qquad y\in\{-1,1\},
\]
when \(c_\star>0\), an observation with \(0\leq c\leq c_\star\) is obtained
from quality \(c_\star\) by a binary channel with conditional mean
\(c/c_\star\).  If \(c_\star=0\), all available qualities are zero and the
corresponding experiments are identical.
A Gaussian observation with variance \(\sigma_\star^2\) dominates one with
variance \(\sigma^2\geq\sigma_\star^2\) by addition of independent Gaussian
noise.  A Poisson count with mean \(\lambda_\star\theta\) dominates one with
mean \(\lambda\theta\), for \(0\leq\lambda\leq\lambda_\star\), by binomial
thinning.

\begin{example}[Failure without a greatest experiment]
\label{ex:incomparable}
Let \(\Theta=\{1,2,3\}\).  Experiment \(a\) reports whether \(\theta=1\),
and experiment \(b\) reports whether \(\theta=2\).  Neither dominates the
other.  With two observations, using each experiment once identifies
\(\theta\) exactly, whereas repeating either experiment leaves one pair of
parameters indistinguishable.  Under the uniform prior and zero-one loss,
the mixed allocation has Bayes risk zero and either repeated experiment has
Bayes risk \(1/3\).  Thus a menu without a greatest element need not collapse,
even when the context is fixed and no adaptation is used.
\end{example}

Example~\ref{ex:incomparable} also explains why the present scalar-chain
assumption is substantive.  In partially ordered or vector-parameter
problems, useful allocations may combine complementary experiments, and menu
cardinality alone does not reduce the design to a covering problem.

\section{Fixed-menu local risk and induced geometry}
\label{sec:fixed-menu}

Fix an interior point \(\theta_0\).  For each context-action pair, suppose
that the one-observation model is differentiable in quadratic mean at
\(\theta_0\), with Fisher information \(I(q,a)\), and that its local
experiments are LAN.  When \(I(q,a)>0\), assume that there is an estimator
\(\widehat\theta_{n,q,a}\) that is compact-uniformly locally efficient: for
every \(H<\infty\),
\[
 \sup_{|h|\leq H}\left|
 \E_{\theta_0+h/\sqrt n,q,a}
 \left[n\left\{\widehat\theta_{n,q,a}
       -(\theta_0+h/\sqrt n)\right\}^2\right]
 -I(q,a)^{-1}\right|\longrightarrow0.
\]
This condition holds for the usual efficient estimators on compact interior
submodels of the three regular families considered below.  Pointwise
efficiency alone would not suffice for the upper bound.  Let \(\psi_q\) be
differentiable at \(\theta_0\), and write
\[
 b(q)=\dot\psi_q(\theta_0),\qquad
 \psi_{q,n}(h)=\sqrt n\{\psi_q(\theta_0+h/\sqrt n)
                         -\psi_q(\theta_0)\}.
\]
If \(T\) estimates \(\psi_{q,n}(h)\), then
\(\psi_q(\theta_0)+T/\sqrt n\) is the associated estimator on the original
scale, whose mean squared error is \(n^{-1}\) times the risk below.  Thus all
local risk differences in this section are \(n\)-scaled.
For a menu \(V\), define the localized minimax risk
\[
 \cR_{n,H}(V,q)=
 \inf_{\Pi,T}\sup_{|h|\leq H}
 \E_{\theta_0+h/\sqrt n,q}^{\Pi}
 \{T-\psi_{q,n}(h)\}^2,
\]
where the infimum is over admissible menu policies \(\Pi\) and terminal
estimators \(T\).  Put
\[
 L(V,q)=\lim_{H\to\infty}\liminf_{n\to\infty}\cR_{n,H}(V,q).
\]

We assume throughout this section that every finite menu has, at each
context, a measurable greatest action \(a_V(q)\), and that the full action
space has an attained measurable greatest action \(a_\star(q)\).  Under DQM,
Blackwell dominance implies Fisher-information monotonicity because the score
of a garbled experiment is the conditional expectation of the original
score.  Hence
\[
 I_V(q):=I(q,a_V(q))=\max_{a\in V}I(q,a),\qquad
 I_\star(q):=I(q,a_\star(q))=\sup_{a\in\cA}I(q,a).
\]
Set \(L_\star(q)=L(\{a_\star(q)\},q)\).
For the global frontier below, assume \(I_\star(q)>0\) for every
\(q\in\cQ\).

\begin{theorem}[Fixed-menu local frontier]
\label{thm:fixed-menu}
Under the preceding regularity and ordered-menu assumptions, for every
\(q\in\cQ\),
\begin{equation}
 L(V,q)=
 \begin{cases}
 b(q)^2/I_V(q),&b(q)\ne0,\\
 0,&b(q)=0,
 \end{cases}
 \qquad
 L_\star(q)=
 \begin{cases}
 b(q)^2/I_\star(q),&b(q)\ne0,\\
 0,&b(q)=0,
 \end{cases}
 \label{eq:fixed-menu-risk}
\end{equation}
with the convention \(1/0=\infty\) in the nonzero-derivative branch.
Consequently, define
\begin{equation}
 D(q,a)=
 \begin{cases}
 0,&b(q)=0,\\
 b(q)^2\{I(q,a)^{-1}-I_\star(q)^{-1}\},
      &b(q)\ne0,\ I(q,a)>0,\\
 \infty,&b(q)\ne0,\ I(q,a)=0,
 \end{cases}
 \label{eq:stat-distortion}
\end{equation}
the optimized worst-context local excess is
\begin{equation}
 A_k=\inf_{V\subset\cA,\,1\leq |V|\leq k}
       \sup_{q\in\cQ}\min_{a\in V}D(q,a).
 \label{eq:Ak}
\end{equation}
If \(\cQ\) and \(\cA\) are compact and \(D\) is lower semicontinuous, the
infimum is attained.
\end{theorem}

\begin{proof}
If \(b(q)=0\), differentiability gives
\[
 \sup_{|h|\leq H}|\psi_{q,n}(h)|=o(1)
\]
for every fixed \(H\).  The constant estimator \(T=0\) therefore gives
\(L(V,q)=L_\star(q)=0\), irrespective of the available information.  This
also proves the zero branch of \eqref{eq:stat-distortion}.  Hence suppose
\(b(q)\ne0\).

Theorem~\ref{thm:collapse} makes every menu policy a garbling of
\(\cE_{q,a_V(q)}^{\otimes n}\), while repetition of \(a_V(q)\) is admissible.
The scalar local asymptotic minimax theorem gives the lower bound in
\eqref{eq:fixed-menu-risk} \citep[Chapter~8]{vanDerVaart1998}.  For the upper
bound, suppose \(I_V(q)>0\), use the compact-uniformly efficient estimator for
\(a_V(q)\), and set
\[
 T_n=b(q)\sqrt n\{\widehat\theta_{n,q,a_V(q)}-\theta_0\}.
\]
Differentiability gives
\[
 \sup_{|h|\leq H}|\psi_{q,n}(h)-b(q)h|=o(1)
\]
for every fixed \(H\).  Compact-uniform local efficiency and
Cauchy--Schwarz therefore yield
\[
 \sup_{|h|\leq H}
 \E_{\theta_0+h/\sqrt n,q,a_V(q)}
 \{T_n-\psi_{q,n}(h)\}^2
 =\frac{b(q)^2}{I_V(q)}+o(1).
\]
When \(I_V(q)=0\), DQM implies that the score of
\(\cE_{q,a_V(q)}\) has zero \(L^2\)-norm.  Hence, for each fixed \(H\),
the one-observation squared Hellinger distances from
\(P_{\theta_0,q,a_V(q)}\) to
\(P_{\theta_0\pm H/\sqrt n,q,a_V(q)}\) are \(o(n^{-1})\).
Tensorization and the triangle inequality show that the total variation
distance between the corresponding \(n\)-fold product laws tends to zero.
On the other hand, differentiability gives
\[
 \left|\psi_{q,n}(H)-\psi_{q,n}(-H)\right|
 =2|b(q)|H+o(1).
\]
The two-point testing bound therefore yields
\(\liminf_n\cR_{n,H}(V,q)\geq c b(q)^2H^2\) for a universal
\(c>0\).  Letting \(H\to\infty\) proves the extended-value formula without
invoking the positive-information local asymptotic minimax theorem.  The
oracle formula is identical.  For a fixed menu, the greatest action maximizes
Fisher information and therefore minimizes \eqref{eq:stat-distortion}; the
same conclusion is immediate in the zero-derivative branch.
Taking the supremum over contexts and the infimum over menus proves
\eqref{eq:Ak}.  For attainment, represent a menu by an ordered \(k\)-tuple,
allowing repeated coordinates, and use compactness and lower semicontinuity.
\end{proof}

The expression in \eqref{eq:Ak} resembles a covering criterion, but its
geometry is not imposed in advance.  It is generated by the local decision
problem.  The appropriate action space may also be a quotient, since
different physical actions can induce the same local experiment.

For the geometric comparison, assume that \(\cQ\) is compact.  Let
\((\cM,d)\) be a compact metric space and let
\(q_{\cA}:\cA\to\cM\) be a surjection.  Suppose
\[
 D(q,a)=\overline D(q,q_{\cA}(a))
\]
for a continuous \(\overline D:\cQ\times\cM\to[0,\infty)\).  Assume that
each context has a unique oracle class \(m(q)\in\cM\), so that
\(\overline D(q,z)=0\) if and only if \(z=m(q)\).  Write
\(S=m(\cQ)\).  The maximum theorem implies that \(m\) is continuous, so
\(S\) is compact.  The surjection permits every quotient menu to be lifted
to a physical menu without increasing its cardinality, and therefore
\[
 A_k=\inf_{V\subset\cM,\,1\leq |V|\leq k}
       \sup_{q\in\cQ}\min_{z\in V}\overline D(q,z).
\]

Suppose next that \(\cM\) is a smooth Riemannian manifold near \(S\), and
that, on a neighborhood of \(S\), \(d\) agrees with the geodesic distance
induced by this Riemannian structure.  Suppose also that
\(\overline D(q,z)\) is three times continuously differentiable in \(z\)
near the graph of \(m\).  Assume
\begin{equation}
 \nabla_z\overline D(q,m(q))=0,
 \qquad
 G_q:=\nabla_z^2\overline D(q,m(q))\succ0,
 \label{eq:hessian-metric}
\end{equation}
with the eigenvalues of \(G_q\) uniformly bounded above and away from zero.
Finally, require uniform separation from the oracle graph:
\begin{equation}
 \eta(\eps):=
 \inf\{\overline D(q,z):d(z,m(q))\geq\eps\}>0
 \quad\text{for every }\eps>0.
 \label{eq:separation}
\end{equation}
This condition follows from compactness, continuity, and uniqueness when the
whole quotient space is compact.

For a compact set \(S\subset\cM\), let
\[
 \rho_k(S)=\inf_{V\subset\cM,\,1\leq |V|\leq k}
            \sup_{s\in S}\min_{z\in V}d(s,z)
\]
be its unrestricted-center covering radius.

\begin{theorem}[Statistical geometry of the menu frontier]
\label{thm:geometry}
Under \eqref{eq:hessian-metric}--\eqref{eq:separation}, there are constants
\(0<c_D\leq C_D<\infty\) such that, for all sufficiently large \(k\),
\begin{equation}
 c_D\rho_k(S)^2\leq A_k\leq C_D\rho_k(S)^2.
 \label{eq:covering-comparison}
\end{equation}
If \(S\) is Ahlfors regular of dimension \(r>0\), then
\begin{equation}
 A_k\asymp k^{-2/r}.
 \label{eq:high-resolution}
\end{equation}
No sharp high-resolution constant is asserted.
\end{theorem}

\begin{proof}
In normal coordinates at \(m(q)\), Taylor expansion and the uniform Hessian
bounds give constants \(\rho_0,c_D,C_D>0\) such that
\begin{equation}
 c_Dd\{z,m(q)\}^2\leq\overline D(q,z)
 \leq C_Dd\{z,m(q)\}^2
 \label{eq:local-quadratic-comparison}
\end{equation}
whenever \(d\{z,m(q)\}\leq\rho_0\), uniformly in \(q\).  Compactness gives
\(\rho_k(S)\to0\), so an asymptotically optimal metric covering and the upper
bound in \eqref{eq:local-quadratic-comparison} prove the upper inequality in
\eqref{eq:covering-comparison}.  This also yields \(A_k\to0\).  By
\eqref{eq:separation}, every distortion-minimizing point of a nearly optimal
menu then lies in the \(\rho_0\)-neighborhood of the relevant oracle class.
The lower inequality in \eqref{eq:local-quadratic-comparison}, followed by the
menu infimum, proves the other half of \eqref{eq:covering-comparison}.

For completeness, let \(\mu\) be an Ahlfors-regular measure on \(S\), so that
\(cR^r\leq\mu\{B(s,R)\}\leq CR^r\) at small radii.  A covering by \(k\)
ambient balls of radius \(R\) may be recentered at points of \(S\) whenever
they intersect \(S\), increasing their radii by at most a factor of two.
The upper volume bound then gives \(\rho_k(S)\geq c'k^{-1/r}\).  A maximal
separated set and the lower volume bound give a \(C'k^{-1/r}\)-net.  Hence
\(\rho_k(S)\asymp k^{-1/r}\), which proves \eqref{eq:high-resolution}.
\end{proof}

The three ordered models give transparent forms of \(D\) on their
positive-information branches; the zero-derivative and zero-information
cases use \eqref{eq:stat-distortion}.  For binary
attenuation with oracle quality one,
\[
 I_c(\theta_0)=\frac{c^2}{1-c^2\theta_0^2},\qquad
 D(q,a)=b(q)^2\{c(q,a)^{-2}-1\}.
\]
For Gaussian location with action-dependent variance,
\[
 D(q,a)=b(q)^2\{\sigma(q,a)^2-\sigma_\star(q)^2\}.
\]
For a Poisson count with mean \(\lambda(q,a)\theta\),
\[
 D(q,a)=b(q)^2\theta_0
 \{\lambda(q,a)^{-1}-\lambda_\star(q)^{-1}\}.
\]
The resulting quotient metrics need not agree globally.  Their common local
quadratic behavior is a consequence of the inverse-information Hessian.

\section{Differential minimax transfer}
\label{sec:differential}

Theorem~\ref{thm:fixed-menu} takes the statistical limit with the menu fixed.
To allow \(k=k_n\) to grow, the approximation to local minimax risk must be
accurate relative to the shrinking fixed-menu excess risk.  This section gives a
sufficient condition in terms of differentiated Bayes risks.

For \(v>0\), let \(\cG_H(v)\) be the bounded Gaussian location experiment
\begin{equation}
 Y=h+\sqrt v Z,\qquad Z\sim N(0,1),\qquad |h|\leq H.
 \label{eq:gaussian-reference}
\end{equation}
Write \(R_H^G(v)\) for its minimax squared-error risk, \(r_v^G(\pi)\) for
the Bayes risk under a prior \(\pi\), and \(w_v^G(Y)\) for the posterior
variance.  The following properties will be used.

\begin{lemma}[Gaussian reference]
\label{lem:gaussian-reference}
Uniformly for \(v\) in compact subsets of \((0,\infty)\),
\begin{align}
 0&\leq R_H^G(v_2)-R_H^G(v_1)\leq v_2-v_1,
       &&0<v_1<v_2,                                      \label{eq:G-Lipschitz}\\
 R_H^G(v)/v&\longrightarrow1,
       &&H\to\infty,                                    \label{eq:G-minimax-limit}\\
 \partial_v r_v^G(\pi)
   &=\E_\pi\{w_v^G(Y)/v\}^2.                           \label{eq:G-mmse-derivative}
\end{align}
Moreover, \(|\partial_v^2r_v^G(\pi)|\leq C(1+H^6)\) for every prior on
\([-H,H]\).  If
\begin{equation}
 R_H^G(v)-r_v^G(\pi_H)\leq\eps_H,\qquad
 H^6\eps_H\to0,
 \label{eq:near-lf}
\end{equation}
then
\begin{equation}
 \E_{\pi_H}\left\{w_v^G(Y)/v-1\right\}^2\longrightarrow0.
 \label{eq:gaussian-pvc}
\end{equation}
\end{lemma}

The first inequality follows by Gaussian convolution and scaling.  A
cosine-squared prior and the van Trees inequality \citep{GillLevit1995} give
\(R_H^G(v)\geq\{v^{-1}+\pi^2/H^2\}^{-1}\), proving
\eqref{eq:G-minimax-limit}.  Identity \eqref{eq:G-mmse-derivative} and the
curvature bound are standard Gaussian MMSE identities
\citep{GuoShamaiVerdu2005,GuoWuShamaiVerdu2011}.  A proof of the final
assertion is included in Appendix~\ref{app:gaussian-pvc}.  Its role is
important: least favorability, together with the one-Lipschitz minimax value,
forces the normalized posterior variance to concentrate near one.

\begin{proposition}[Posterior Dirichlet identity]
\label{prop:posterior-dirichlet}
Let \(U\) be the unknown quantity, let \(G=g(U)\) be a bounded target under an
arbitrary prior, and let \(X_s\)
be obtained from \(X_0\) by a parameter-free Markov degradation semigroup.
Write \(m_s(x)=\E(G\mid X_s=x)\) and
\(r_s^g=\E\Var(G\mid X_s)\).
\begin{enumerate}
\item Suppose the semigroup has a finite or countable jump generator
\[
 Lf(x)=\sum_zq(x,z)\{f(z)-f(x)\},
\]
and write \(\mu_s(x)=\Pp(X_s=x)\).  At the differentiation point, assume
\(\mu_s(z)>0\) whenever \(\mu_s(x)q(x,z)\neq0\) for some \(x\), and suppose
the displayed sums are absolutely convergent.  Then
\begin{equation}
 \partial_sr_s^g
 =\E\sum_zq(X_s,z)\{m_s(z)-m_s(X_s)\}^2.
 \label{eq:jump-dirichlet}
\end{equation}
\item For Brownian convolution \(X_s=X_0+\sqrt{s}\,Z\), under the usual
differentiation and boundary conditions,
\begin{equation}
 \partial_sr_s^g=\E\|\nabla m_s(X_s)\|^2.
 \label{eq:heat-dirichlet}
\end{equation}
\end{enumerate}
\end{proposition}

The proof is in Section S2 of the Supplementary Material.  Coordinate-flip
rates in \eqref{eq:jump-dirichlet} give the binary hypercube identity; the
pure-death rates \(q(j,j-1)=j\) give the Poisson identity; and
\eqref{eq:heat-dirichlet} gives the Gaussian posterior-covariance formula.
This posterior-energy identity supplies the exact differentiation step.  The
all-prior approximations \eqref{eq:GT2}--\eqref{eq:GT3}, which are needed for
minimax transfer, remain separate model-specific statements.

We now state the transfer condition.  Let \(\cK\Subset(0,\infty)\), and
consider local experiments
\[
 \cE_{n,H}(v)=\{P_{n,v,h}:|h|\leq H\},\qquad v\in\cK,
\]
for the identity target \(h\).  Denote their minimax and Bayes risks by
\(R^E_{n,H}(v)\) and \(r^E_{n,v}(\pi)\).

\begin{assumption}[Differentiated Gaussian transfer]
\label{ass:GT}
Uniformly over \(v\in\cK\), all priors on \([-H,H]\), and the stated range
of \(H=H_n\), suppose the following.
\begin{enumerate}
\item Bayes--minimax duality holds:
\[
 R^E_{n,H}(v)=\sup_{\pi}r^E_{n,v}(\pi),
\]
with an attaining least-favorable prior.
 \item Every Bayes risk is differentiable in \(v\), and
 \begin{equation}
  0\leq \partial_vr^E_{n,v}(\pi)\leq C_{\cK}H^4.
  \label{eq:GT1}
 \end{equation}
\item There are deterministic errors \(\eta_{0,n}(H)\) and
\(\eta_{1,n}(H)\) such that
\begin{align}
 |r^E_{n,v}(\pi)-r_v^G(\pi)|
   &\leq\eta_{0,n}(H),                                  \label{eq:GT2}\\
  \left|\partial_vr^E_{n,v}(\pi)
         -\E^G\{w_v^G(Y)/v\}^2\right|
    &\leq\eta_{1,n}(H).                                  \label{eq:GT3}
 \end{align}
 \end{enumerate}
 \end{assumption}

A representation of the derivative as the expectation of a uniformly bounded
square is a sufficient, but not necessary, way to verify \eqref{eq:GT1}.
Assumption~\ref{ass:GT} requires only the nonnegative uniform derivative bound
and the direct comparison in \eqref{eq:GT3}.

The uniformity over priors in \eqref{eq:GT2}--\eqref{eq:GT3} is essential.
Pointwise convergence of likelihood ratios, asymptotic normality of one
estimator, and ordinary LAN do not give this assumption.

The all-prior conditions can nevertheless be obtained from prior-free
properties of the sampling laws.  Suppose that increasing \(v\) is a
parameter-free Markov degradation, and that a sufficient statistic admits a
Gaussian embedding.  The primitive conditions in Section S5 of the
Supplementary Material require a multiplicative local Gaussian likelihood
comparison with errors \(\eps_n,\kappa_n,\tau_n,b_n\), together with an
edge-score expansion, Fisher calibration, and energy-tail bounds measured by
\(\gamma_n,\chi_n,\tau_n^\Gamma\).  None of these quantities depends on a
prior or posterior distribution.

\begin{theorem}[Primitive likelihood-generator transfer]
\label{thm:primitive-transfer}
Assume the prior-free likelihood-generator conditions PLG1--PLG4 in Section S5
of the Supplementary Material, compact Bayes--minimax duality, and the
posterior Dirichlet identity in Proposition~\ref{prop:posterior-dirichlet}.
Then \eqref{eq:GT2}--\eqref{eq:GT3} hold uniformly over all priors, with
\begin{align}
 \eta_{0,n}(H)
 &\leq C\left[H^2\eps_n(H)+\kappa_n(1+H^3)
               +H^2\{\tau_n(H)+b_n(H)\}\right],
 \label{eq:primitive-eta0}\\
 \eta_{1,n}(H)
 &\leq C\left[H^4\eps_n(H)+\kappa_n(1+H^5)\right. \notag\\
 &\hspace{17mm}\left.
               +H^4\{\tau_n(H)+b_n(H)+\gamma_n(H)\}
               +\chi_n(H)+\tau_n^\Gamma(H)\right].
 \label{eq:primitive-eta1}
\end{align}
Moreover,
\[
 0\leq\partial_vr^E_{n,v}(\pi)
 \leq C_{\cK}H^4+C\eta_{1,n}(H).
\]
Consequently, if \(H\geq1\) and \(\eta_{1,n}(H)\to0\), the primitive
conditions imply Assumption~\ref{ass:GT} after enlarging \(C_{\cK}\).
\end{theorem}

The proof of Theorem~\ref{thm:primitive-transfer}, including the edge-tilt and
quadrature arguments behind \eqref{eq:primitive-eta0}--
\eqref{eq:primitive-eta1}, is given in Section S5 of the Supplementary
Material.  The theorem is not a consequence of ordinary LAN: its generator
calibration controls the derivative of Bayes risk along the degradation
coordinate.

\begin{theorem}[Differential minimax transfer]
\label{thm:differential-transfer}
Let \(\delta_n\downarrow0\), and suppose
\([v_n,v_n+\delta_n]\) remains in a compact interior subset of \(\cK\).
If \(H_n\to\infty\) and
\begin{equation}
 H_n^6\eta_{0,n}(H_n)\to0,\qquad
 \eta_{1,n}(H_n)\to0,\qquad
 H_n^{10}\delta_n\to0,
 \label{eq:transfer-rates}
\end{equation}
then, under Assumption~\ref{ass:GT},
\begin{equation}
 R^E_{n,H_n}(v_n+\delta_n)-R^E_{n,H_n}(v_n)
 =\delta_n\{1+o(1)\}.
 \label{eq:differential-transfer}
\end{equation}
The conclusion is uniform when the constants in
Assumption~\ref{ass:GT} are uniform.
\end{theorem}

\begin{proof}
Let \(\pi_{n,v}\) be least favorable for \(\cE_{n,H_n}(v)\).  Since
\eqref{eq:GT2} is uniform over priors, it also compares the two minimax
values, and hence
\[
 R_{H_n}^G(v)-r_v^G(\pi_{n,v})
 \leq2\eta_{0,n}(H_n).
\]
Lemma~\ref{lem:gaussian-reference}, \eqref{eq:GT3}, and the first two
conditions in \eqref{eq:transfer-rates} give
\(\partial_vr^E_{n,v}(\pi_{n,v})=1+o(1)\).

It remains to use an endpoint prior at intermediate variance levels.
Equation \eqref{eq:GT1} implies that every Bayes risk, and therefore the
minimax value, is nondecreasing and \(C H_n^4\)-Lipschitz in \(v\).  A prior
least favorable at either endpoint is consequently
\(C H_n^4\delta_n\)-least favorable throughout the interval.  The last
condition in \eqref{eq:transfer-rates} makes this error admissible in
\eqref{eq:near-lf}.  Thus both endpoint priors have derivative \(1+o(1)\)
uniformly over the interval.  Integrating their Bayes-risk derivatives and
evaluating each minimax value under the opposite endpoint prior sandwiches
the minimax increment between two quantities equal to
\(\delta_n\{1+o(1)\}\).
\end{proof}

The following deterministic choice will be used in all non-Gaussian examples.

\begin{corollary}[Slow localization radius]
\label{cor:slow-radius}
Suppose, uniformly on \(\cK\),
\begin{align}
 \eta_{0,n}(H)&\leq C H^5n^{-1/2}+CH^2e^{-cH^2},
                                                        \label{eq:eta0-rate}\\
 \eta_{1,n}(H)&\leq C H^7n^{-1/2}+CH^4e^{-cH^2}.
                                                        \label{eq:eta1-rate}
\end{align}
For every \(\delta_n\downarrow0\), the choice
\begin{equation}
 H_n=\min\{\delta_n^{-1/20},n^{1/44}\}
 \label{eq:slow-radius}
\end{equation}
satisfies \eqref{eq:transfer-rates} and therefore
\eqref{eq:differential-transfer}.
\end{corollary}

\begin{proof}
The choice gives \(H_n\to\infty\),
\(H_n^{11}/\sqrt n\to0\), and
\(H_n^{10}\delta_n\to0\).  The exponential terms are negligible.
\end{proof}

The menu problem concerns the smooth local target \(\psi_{q,n}\), rather
than the identity target.  For fixed \(q\) and \(v\), write
\[
 \cR_{n,H}^{\psi}(q,v)=
 \inf_T\sup_{|h|\leq H}
 \E_{n,q,v,h}\{T-\psi_{q,n}(h)\}^2
\]
for its minimax risk.  We impose one further transfer condition.

\begin{assumption}[Smooth-target transfer]
\label{ass:ST}
Uniformly over contexts, priors, and relevant inverse-information levels,
assume \(0<b_0\leq|b(q)|\leq b_1<\infty\), uniformly bounded second
derivatives of \(\psi_q\), and, for every stated \(n,H,q,v\),
Bayes--minimax duality for the smooth target:
\[
 \cR_{n,H}^{\psi}(q,v)
 =\sup_{\pi\in\cP([-H,H])}r^{\psi}_{n,v}(\pi),
\]
where the supremum is attained.  In addition, suppose
\begin{align}
 |r^{\psi}_{n,v}(\pi)-b(q)^2r^h_{n,v}(\pi)|
 &\leq C\{H^3n^{-1/2}+H^4n^{-1}\},                     \label{eq:ST1}\\
 \{\,|b(q)|-CHn^{-1/2}\,\}_+^2\partial_vr^h_{n,v}(\pi)
 &\leq \partial_vr^{\psi}_{n,v}(\pi)
 \leq \{\,|b(q)|+CHn^{-1/2}\,\}^2\partial_vr^h_{n,v}(\pi).
                                                               \label{eq:ST2}
\end{align}
\end{assumption}

The risk comparison \eqref{eq:ST1} follows from Taylor's theorem and
conditional variance.  Condition \eqref{eq:ST2} is stronger: the Dirichlet
form that differentiates Bayes risk must be generated by posterior slopes.
Section S5 of the Supplementary Material shows that it follows from monotone
edge likelihood ratios and the stated target smoothness.  It holds in each
model considered in Section~\ref{sec:models}.

We now return to the context-indexed action family.  Assume that its local
experiments form a continuous Blackwell chain in inverse information:
\begin{equation}
 v_\star(q)\leq v(q,a),\qquad
 \cE_q(v_1)\succeq_B\cE_q(v_2)\quad\text{if }v_1\leq v_2,
 \label{eq:blackwell-chain}
\end{equation}
and every intermediate inverse-information level in a fixed neighborhood of
\(v_\star(q)\) is available.  For a menu,
\(v_V(q)=\min_{a\in V}v(q,a)\).  In this coordinate, the frontier in
\eqref{eq:Ak} is
\begin{equation}
 A_k=\inf_{1\leq |V|\leq k}\sup_{q\in\cQ}
 b(q)^2\{v_V(q)-v_\star(q)\}.
 \label{eq:Ak-v}
\end{equation}
With this notation, the finite-neighborhood menu frontier is
\begin{equation}
 F_{n,k}(H)=\inf_{1\leq |V|\leq k}\sup_{q\in\cQ}
 \left\{\cR_{n,H}^{\psi}(q,v_V(q))
       -\cR_{n,H}^{\psi}(q,v_\star(q))\right\}.
 \label{eq:finite-frontier}
\end{equation}
Thus the two finite-neighborhood risks are compared at the same context
before taking the context supremum.

\begin{theorem}[Joint sample-menu law]
\label{thm:joint-law}
Let \(\cQ\) be compact, suppose \eqref{eq:blackwell-chain} holds, and assume
that, for some \(\bar\delta>0\), the intervals
\([v_\star(q),v_\star(q)+\bar\delta]\) lie in a common compact subset of
\((0,\infty)\).  Assume Assumptions~\ref{ass:GT} and \ref{ass:ST} uniformly
over all contexts and inverse-information levels in these intervals.  Suppose
\(A_k\to0\), and that the infimum in \eqref{eq:Ak-v} is attained or
approached with relative error \(o(1)\).  For any \(k_n\to\infty\) such that
\(A_{k_n}>0\), let \(H_n\) satisfy
\begin{equation}
 H_n\to\infty,\qquad
 \frac{H_n^{11}}{\sqrt n}\to0,\qquad
 H_n^{10}A_{k_n}\to0.
 \label{eq:joint-radius-rates}
\end{equation}
If \eqref{eq:eta0-rate}--\eqref{eq:eta1-rate} hold, then
\begin{equation}
 F_{n,k_n}(H_n)=A_{k_n}\{1+o(1)\},
 \label{eq:joint-law}
\end{equation}
where \(F_{n,k}\) is defined in \eqref{eq:finite-frontier}.  The
conclusion has no additional sample-size/menu-size restriction.  Conditions
\eqref{eq:joint-radius-rates} can always be met; one explicit choice is
\begin{equation}
 H_n=\min\{A_{k_n}^{-1/20},n^{1/44}\}.
 \label{eq:joint-radius}
\end{equation}
\end{theorem}

\begin{proof}
First consider a gap \(0\leq\delta\leq CA_{k_n}\).  By
Assumption~\ref{ass:ST}, take an attaining least-favorable prior for the
smooth target.  Equation~\eqref{eq:ST1} makes it
\(C\{H_n^3/\sqrt n+H_n^4/n\}\) of least favorable for the identity target,
after division by \(b(q)^2\).  Multiplication by \(H_n^6\) sends this
certificate error to zero under \eqref{eq:joint-radius-rates}: eventually
\(H_n\geq1\), so both \(H_n^9/\sqrt n\) and \(H_n^{10}/n\) vanish.  The
argument in the proof of
Theorem~\ref{thm:differential-transfer} therefore gives identity-target
posterior-variance concentration for the endpoint priors of the smooth-target
problem.  Equation \eqref{eq:ST2} and the same endpoint-prior sandwich yield,
uniformly in \(q\),
\begin{equation}
 \cR_{n,H_n}^{\psi}(q,v+\delta)
 -\cR_{n,H_n}^{\psi}(q,v)
 =b(q)^2\delta\{1+o(1)\}.
 \label{eq:smooth-differential}
\end{equation}
Conditions \eqref{eq:joint-radius-rates} imply the transfer conditions
uniformly for these gaps: the error bounds give
\(H_n^6\eta_{0,n}(H_n)\to0\) and
\(\eta_{1,n}(H_n)\to0\), while
\(H_n^{10}\delta\to0\) for \(\delta\leq C A_{k_n}\).  Since
\(A_{k_n}\to0\) and \(|b(q)|\geq b_0\), all gaps used in
the upper bound and all intermediate gaps used in the lower bound eventually
belong to the local interval of Theorem~\ref{thm:joint-law}.

For the upper bound, take a menu whose worst distortion is
\(A_{k_n}\{1+o(1)\}\) and apply \eqref{eq:smooth-differential} at each
context.  For the lower bound, fix a menu.  For any deterministic
\(\eps_n\downarrow0\), some context has distortion at least
\((1-\eps_n)A_{k_n}\).  If the menu gap at that context is larger than the
corresponding value, insert on the continuous chain an intermediate
experiment having exactly that inverse-information gap.  Blackwell
monotonicity places the menu risk above the intermediate risk, and
\eqref{eq:smooth-differential} gives
\((1-\eps_n)A_{k_n}\{1+o(1)\}\).  Taking the menu infimum and then sending
\(\eps_n\) to zero proves \eqref{eq:joint-law}.
\end{proof}

\begin{corollary}[Joint law under primitive conditions]
\label{cor:primitive-joint}
Suppose the conditions of Theorem~\ref{thm:joint-law} concerning the compact
context space, the Blackwell chain, the frontier, and relative menu
approximation hold.  For the identity target, assume compact Bayes--minimax
duality with attainment, the posterior Dirichlet identity in
Proposition~\ref{prop:posterior-dirichlet}, and PLG1--PLG4 uniformly over
contexts and the relevant inverse-information levels.  Suppose the right-hand
sides of \eqref{eq:primitive-eta0}--\eqref{eq:primitive-eta1} satisfy
\eqref{eq:eta0-rate}--\eqref{eq:eta1-rate}.  For the smooth target, assume
compact Bayes--minimax duality with attainment, uniformly bounded second
derivatives, and \(0<b_0\leq|b(q)|\leq b_1<\infty\).  Assume also that the
active-edge likelihood ratios of the jump generator are monotone in the local
parameter.  Then, for every \(k_n\to\infty\) with
\(A_{k_n}>0\), every \(H_n\) satisfying
\eqref{eq:joint-radius-rates} obeys
\[
 F_{n,k_n}(H_n)=A_{k_n}\{1+o(1)\}.
\]
In particular, the choice in \eqref{eq:joint-radius} is admissible.
\end{corollary}

\begin{proof}
Theorem~\ref{thm:primitive-transfer} verifies
Assumption~\ref{ass:GT} with the required error rates.  For jump generators,
the monotone-tilt secant argument in Section S5 of the Supplementary Material
gives \eqref{eq:ST2}; Taylor expansion and conditional variance give
\eqref{eq:ST1}.  The result now follows from
Theorem~\ref{thm:joint-law}.
\end{proof}

\begin{remark}[Zero frontier]
\label{rem:zero-frontier}
If \(A_k=0\) as an infimum, then \(F_{n,k}(H)=0\) as an infimum for fixed
\(n,H\).  Indeed, \eqref{eq:GT1} and \eqref{eq:ST2} give a finite local
Lipschitz constant in \(v\), and menus with arbitrarily small distortion
give arbitrarily small finite-sample excess.  An oracle-exact finite menu is
asserted only when the zero frontier is attained.
\end{remark}

\begin{remark}[Why LAN is insufficient]
\label{rem:lan-insufficient}
Theorem~\ref{thm:joint-law} permits \(A_{k_n}\) to vanish faster than any
prespecified polynomial rate.  An absolute LAN approximation cannot then be
divided by \(A_{k_n}\).  Assumption~\ref{ass:GT} supplies a differentiated,
all-prior approximation on the relevant least-favorable priors.  It is an
additional model property, not a reformulation of LAN.
\end{remark}

\begin{proposition}[A counterexample based on pointwise Gaussian limits]
\label{prop:lan-witness}
Fix \(v_0>0\), \(0<\gamma<1\), and \(a_n\downarrow0\).  On a compact
neighborhood of \(v_0\), define
\[
 \widetilde v_n(v)
 =v+\gamma a_n\tanh\{(v-v_0)/a_n\},
 \qquad
 Y_{n,v}=h+\sqrt{\widetilde v_n(v)}\,Z,
 \quad h\in\R.
\]
Then \(v\mapsto\widetilde v_n(v)\) is increasing, and for every fixed \(v\)
the experiment converges, uniformly over translations \(h\), to the Gaussian
location experiment of variance \(v\).  Nevertheless, for every
\(\delta_n>0\) with \(\delta_n/a_n\to0\),
\begin{equation}
 \frac{R_{n,\infty}(v_0+\delta_n)-R_{n,\infty}(v_0)}{\delta_n}
 \longrightarrow1+\gamma,
 \label{eq:lan-witness}
\end{equation}
where \(R_{n,\infty}(v)\) is the unrestricted Gaussian-location minimax risk.
The same conclusion holds for the minimax risk \(R_{n,H_n}(v)\) on
\([-H_n,H_n]\) whenever \(H_n\to\infty\) and
\(H_n^{-2}=o(\delta_n)\).
\end{proposition}

\begin{proof}
The derivative of \(\widetilde v_n\) is
\(1+\gamma\operatorname{sech}^2\{(v-v_0)/a_n\}>0\).
For fixed \(v\), \(|\widetilde v_n(v)-v|\leq\gamma a_n\), so Gaussian
total variation tends to zero; translation invariance makes the comparison
uniform in \(h\).  Under squared-error loss on \(\R\), the minimax risk of a
Gaussian location experiment equals its noise variance.  Hence the left-hand
side of \eqref{eq:lan-witness} is
\[
 1+\gamma\frac{a_n}{\delta_n}
       \tanh(\delta_n/a_n)\longrightarrow1+\gamma.
\]
For the bounded problem, the Gaussian estimator and the cosine-prior van
Trees bound give, uniformly for the variances in this compact neighborhood,
\[
 \widetilde v_n(v)-C H_n^{-2}
 \leq R_{n,H_n}(v)\leq\widetilde v_n(v).
\]
The bounded-risk increment therefore differs from the corresponding
effective-variance increment by \(O(H_n^{-2})=o(\delta_n)\).
\end{proof}

The counterexample does not contradict LAN: it shows that a limiting inverse-
information coordinate can be miscalibrated on the shrinking
\(\delta_n\)-scale while being correct at every fixed level.  Assumption
\ref{ass:GT} rules out precisely this loss of differential calibration.

\section{Verification in ordered experiment families}
\label{sec:models}

We next verify the abstract conditions in four models.  The Gaussian family is
exact.  The binary and Poisson arguments compare posterior distributions
uniformly over all priors on the local interval.  The negative-binomial model
is obtained from Theorem~\ref{thm:primitive-transfer}.  The full local-limit and
primitive-condition proofs are given in Sections S3--S5 of the Supplementary
Material.

\begin{theorem}[Transfer in four ordered families]
\label{thm:model-transfer}
The following statements hold uniformly when the fixed model parameters and
the inverse-information coordinate range over compact interior sets.
\begin{enumerate}
\item In the Gaussian experiment \eqref{eq:gaussian-reference},
Assumption~\ref{ass:GT} holds with zero comparison errors.
\item Fix \(t_0\in(-1,1)\).  In the binary experiment
\begin{equation}
 \Pp_{h,c}(Y=y)=\frac{1+yc(t_0+h/\sqrt n)}2,
 \qquad y\in\{-1,1\},
 \label{eq:binary-model}
\end{equation}
for \(c>0\), put \(x=c^{-2}\) and \(v=x-t_0^2=I_c(t_0)^{-1}\).
If \(1\leq H=o(n^{1/6})\), Assumption~\ref{ass:GT} holds with
\begin{align}
 \eta_{0,n}(H)&\leq C H^5n^{-1/2}+CH^2e^{-c_0H^2},       \label{eq:binary-risk-transfer}\\
 \eta_{1,n}(H)&\leq C H^7n^{-1/2}+CH^4e^{-c_0H^2}.
                                                               \label{eq:binary-deriv-transfer}
\end{align}
\item Fix \(\theta_0>0\).  In the Poisson experiment
\[
 Y_i\sim\operatorname{Poisson}\{(\theta_0+h/\sqrt n)/x\},
 \qquad |h|\leq H,
\]
put \(v=x\theta_0=I_x(\theta_0)^{-1}\).  If
\(1\leq H=o(n^{1/6})\),
Assumption~\ref{ass:GT} holds with the same orders as
\eqref{eq:binary-risk-transfer}--\eqref{eq:binary-deriv-transfer}.
\item Fix \(\theta_0,r>0\).  Let \(Y_1,\ldots,Y_n\) be independent
negative-binomial observations with shape \(r\) and probability generating
function
\[
 \E_{h,x}z^{Y_i}
 =\left\{\frac{xr}{xr+(\theta_0+h/\sqrt n)(1-z)}\right\}^{r},
 \qquad x\geq1.
\]
Thus \(Y_i\) has mean \((\theta_0+h/\sqrt n)/x\).  Put
\(v=x\theta_0+\theta_0^2/r\).  If \(1\leq H=o(n^{1/6})\),
Assumption~\ref{ass:GT} holds with the same orders as
\eqref{eq:binary-risk-transfer}--\eqref{eq:binary-deriv-transfer}.
\end{enumerate}
For every twice continuously differentiable scalar target with derivative
bounded away from zero, all four models also satisfy
Assumption~\ref{ass:ST}.
\end{theorem}

We give the main identities behind the theorem.  In the binary model, let
\(\mu_J\) and \(w_J\) be the posterior mean and variance of \(h\) after
\(n-1\) observations.  Differentiating Bayes risk along the binary symmetric
channel gives the exact formula
\[
 \partial_v r^B_{n,v}(\pi)
 =\E_\pi\left[
   \frac{w_J}{x-(t_0+\mu_J/\sqrt n)^2}
 \right]^2.
\]
The unbiased local statistic
\[
 T_{n,x}(J)=\frac{2J-n}{c\sqrt n}-\sqrt n\,t_0
\]
has limiting Gaussian variance \(v=x-t_0^2\).  A third-order expansion of
the exact binomial likelihood ratio, followed by a uniform lattice local
limit, compares the binary and Gaussian posteriors multiplicatively on
\(|T_{n,x}|\leq C H\).  Exponential tails and a rectangle-rule bound then
give \eqref{eq:binary-risk-transfer}--\eqref{eq:binary-deriv-transfer}.  The
comparison is made after mixing over an arbitrary prior, which is why it
verifies the all-prior requirement in Assumption~\ref{ass:GT}.  The
one-dimensional local limit may alternatively be obtained from the
multinomial-normal comparison of \citet{Carter2002}.

For the Poisson model, the total count \(J=\sum_iY_i\) is sufficient.
With \(\mu_J,w_J\) denoting its posterior mean and variance, differentiation
along the thinning semigroup gives
\[
 \partial_v r^P_{n,v}(\pi)
 =\E_\pi\left[
   \frac{w_J}{v\sqrt{1+\mu_J/(\theta_0\sqrt n)}}
 \right]^2.
\]
The local statistic \(T_{n,x}(J)=(xJ-n\theta_0)/\sqrt n\) has variance
\(v=x\theta_0\) at \(h=0\).  Stirling expansion, uniform Poisson tails, and a
half-lattice rectangle rule yield the same posterior-risk and derivative
orders as in the binary experiment.

For the negative-binomial model, the total count is again sufficient and
binomial thinning gives a parameter-free pure-death degradation.  The
normalized count \(T_{n,x}=(xJ-n\theta_0)/\sqrt n\) has variance
\(v=x\theta_0+\theta_0^2/r\).  Its multiplicative local Gaussian expansion,
edge-score calibration, and weighted energy-tail bound verify the primitive
conditions of Theorem~\ref{thm:primitive-transfer}; see Section S5 of the
Supplementary Material.

For a smooth target \(g=\psi_{q,n}\), clipping decisions to the compact range
of \(g\) leaves the squared-error risk unchanged.  The four experiments are
dominated, their likelihoods are continuous in the local parameter in
\(L^1\), and the local parameter interval is compact.  The compact statistical
minimax theorem therefore gives Bayes--minimax duality.  Bayes risk is upper
semicontinuous on the weakly compact space of priors, so its supremum is
attained.  This verifies the first requirement in
Assumption~\ref{ass:ST}.

The differentiated Bayes risks in the four models are quadratic forms in
adjacent posterior means or in a posterior covariance.  If \(U,U'\) are
conditionally independent posterior draws, then
\[
 \frac{\Cov\{g(U),U\mid Y\}}{\Var(U\mid Y)}
 =
 \frac{\E[\{g(U)-g(U')\}(U-U')\mid Y]}
      {\E[(U-U')^2\mid Y]},
\]
whenever the denominator is nonzero.  The right-hand side is a convex
average of secant slopes.  Taylor's theorem places every slope between
\(b(q)-CH/\sqrt n\) and \(b(q)+CH/\sqrt n\).  When the posterior variance
vanishes, both differentiated risks vanish.  This proves
\eqref{eq:ST2}; \eqref{eq:ST1} follows directly from conditional variance.

\begin{corollary}[Model-specific differential increments]
\label{cor:model-increments}
Let \(\delta_n\downarrow0\) and choose \(H_n\) as in
\eqref{eq:slow-radius}.  In each model of
Theorem~\ref{thm:model-transfer}, uniformly over compact interior base
points,
\[
 R_{n,H_n}(v+\delta_n)-R_{n,H_n}(v)
 =\delta_n\{1+o(1)\}.
\]
For a smooth target \(\psi\), the right-hand side becomes
\(\dot\psi(\theta_0)^2\delta_n\{1+o(1)\}\).
\end{corollary}

The consequences for experiment menus follow by substituting their
inverse-information fields into Theorem~\ref{thm:joint-law}.  The following
counting model gives a nonquantum example in which the frontier is available
in closed form.

\begin{corollary}[Calibrated Poisson sensor menus]
\label{cor:poisson-sensor-menu}
Let \((\cM,d)\) be a compact geodesic Ahlfors-regular metric space of
dimension \(r>0\), and take \(\cQ=\cA=\cM\).  Fix \(\theta_0>0\),
\(x_0\geq1\), and \(\kappa>0\).  After a context \(q\) is disclosed, sensor
calibration \(a\) produces independent counts
\[
 Y_i\sim\operatorname{Poisson}\{\theta/x(q,a)\},
 \qquad
 x(q,a)=x_0+\kappa d(q,a)^2.
\]
Let \(\psi(\theta)\) be twice continuously differentiable near \(\theta_0\),
with \(b=\dot\psi(\theta_0)\ne0\).  Then
\begin{equation}
 A_k=b^2\theta_0\kappa\rho_k(\cM)^2
     \asymp k^{-2/r}.
 \label{eq:poisson-sensor-frontier}
\end{equation}
For every \(k_n\to\infty\) and every \(H_n\) satisfying
\eqref{eq:joint-radius-rates},
\[
 F_{n,k_n}(H_n)=A_{k_n}\{1+o(1)\}.
\]
Consequently, the excess mean squared error for estimating \(\psi(\theta)\)
on the original scale is of order \(n^{-1}k_n^{-2/r}\).
\end{corollary}

\begin{proof}
If \(x_2\geq x_1\), thinning a
\(\operatorname{Poisson}(\theta/x_1)\) count with probability
\(x_1/x_2\) gives a \(\operatorname{Poisson}(\theta/x_2)\) count.  The
kernel does not depend on \(\theta\).  Hence the menu member nearest to \(q\)
is greatest in Blackwell order, and the oracle calibration is \(a=q\).  The
one-observation information and its inverse at \(\theta_0\) are
\[
 I(q,a)=\{\theta_0x(q,a)\}^{-1},
 \qquad v(q,a)=\theta_0x(q,a).
\]
It follows that
\[
 D(q,a)=b^2\theta_0\kappa d(q,a)^2,
\]
which gives the identity in \eqref{eq:poisson-sensor-frontier}.  Ahlfors
regularity gives \(\rho_k(\cM)\asymp k^{-1/r}\).  The geodesic property
provides every intermediate inverse-information level between an action and
the oracle, while compactness keeps all such levels in a common interior
interval.  The Poisson part of Theorem~\ref{thm:model-transfer} and
Theorem~\ref{thm:joint-law} give the joint assertion.  Dividing the localized
risk difference by \(n\) gives the original-scale statement.
\end{proof}

We next record the quantum specialization, which has a non-Euclidean action
space.

For the binary/qubit family, extend the finite-neighborhood risk notation to
the completely uninformative endpoint.  If \(c(q,a)=0\), put
\(v(q,a)=\infty\) and define
\begin{equation}
 \cR_{n,H}^{\psi}(q,\infty)
 :=\frac14\left\{
       \max_{|h|\leq H}\psi_{q,n}(h)
       -\min_{|h|\leq H}\psi_{q,n}(h)
      \right\}^2.
 \label{eq:qubit-zero-information-risk}
\end{equation}

\begin{corollary}[Radial qubit measurement menus]
\label{cor:qubit}
Let \(\rho_{q,t}=\{I+tq^\mathsf{T}\bm\sigma\}/2\) be a radial qubit state,
where the direction \(q\in\mathbb S^2\) is disclosed and the radius \(t\) is
unknown.  A projective measurement along an unoriented axis
\([a]\in\mathbb {RP}^2\) gives the binary experiment
\eqref{eq:binary-model} with
\[
 c(q,a)=|q^\mathsf{T}a|
       =\cos d_{\mathbb {RP}^2}([q],[a]).
\]
For a smooth spectral functional \(\psi(t)\), at an interior
\(t_0\) with \(\dot\psi(t_0)\ne0\),
\begin{equation}
 A_k=\dot\psi(t_0)^2
       \tan^2\{\rho_k(S)\},
 \qquad S=\{[q]:q\in\cQ\}.
 \label{eq:qubit-frontier}
\end{equation}
If \(S\) is Ahlfors regular of dimension \(r\), then
\(A_k\asymp k^{-2/r}\).  For the full Bloch context class
\(S=\mathbb {RP}^2\), the sharp covering asymptotic gives
\[
 kA_k\longrightarrow
 \frac{4\pi}{3\sqrt3}\,\dot\psi(t_0)^2.
\]
For every \(k_n\to\infty\) with
\(A_{k_n}>0\), and every \(H_n\) satisfying
\eqref{eq:joint-radius-rates}, the finite-neighborhood excess risk satisfies
\[
 F_{n,k_n}(H_n)=A_{k_n}\{1+o(1)\}.
\]
In particular, the explicit radius in \eqref{eq:joint-radius} is admissible.
\end{corollary}

\begin{proof}
The Born probabilities are
\(\Pp(Y=y)=\{1+ytq^\mathsf{T}a\}/2\).  Since \(q\) is known, outcome
relabeling replaces \(q^\mathsf{T}a\) by its absolute value.  Binary
attenuation is ordered by a parameter-free binary symmetric channel.
At \(c=0\), the observations are symmetric Bernoulli variables whose law
does not depend on \(h\).  If \(T\) is any terminal rule, its mean and
variance are therefore common to all \(h\), and
\[
 \E\{T-\psi_{q,n}(h)\}^2
 =\Var(T)+\{\E T-\psi_{q,n}(h)\}^2.
\]
Randomization cannot improve the worst-case risk, and the constant midpoint
of the target range attains \eqref{eq:qubit-zero-information-risk}.  For every
fixed \(H\), differentiability and \(\dot\psi(t_0)\ne0\) give
\[
 \cR_{n,H}^{\psi}(q,\infty)
 =\dot\psi(t_0)^2H^2+o(1)
 \qquad(n\to\infty).
\]
Thus its subsequent \(H\to\infty\) limit is infinite, consistently with the
extended-value fixed-menu distortion.  For \(c>0\),
\[
 I_c(t_0)^{-1}-I_1(t_0)^{-1}=c^{-2}-1
 =\tan^2 d_{\mathbb {RP}^2}([q],[a]).
\]
Because \(\tan^2\) is increasing on \([0,\pi/2)\), taking the menu infimum
gives \eqref{eq:qubit-frontier}, with the extended-value convention when a
covering radius equals \(\pi/2\).  This infinite value is the first-order
distortion, not the finite-\(H\) risk in
\eqref{eq:qubit-zero-information-risk}.  If \(S\) is Ahlfors regular, the direct
covering argument in the proof of Theorem~\ref{thm:geometry} gives
\(\rho_k(S)\asymp k^{-1/r}\).  Since \(\rho_k(S)\to0\),
\(\tan^2\{\rho_k(S)\}\asymp\rho_k(S)^2\), which proves the stated rate
without requiring the distortion to be finite on the whole action space.
For \(S=\mathbb {RP}^2\), the geodesic-disc covering asymptotic of
\citet{Gruber1998}, applied to a two-dimensional Riemannian manifold of area
\(2\pi\), gives
\[
 k\rho_k(\mathbb {RP}^2)^2\longrightarrow
 \frac{4\pi}{3\sqrt3}.
\]
The sharp constant follows from \eqref{eq:qubit-frontier} and
\(\tan^2 x=x^2+O(x^4)\) as \(x\to0\).

For the joint limit, sufficiently fine covering menus use only actions in a
fixed neighborhood of their oracle axes, where \(c\) is bounded away from
zero and Theorem~\ref{thm:model-transfer} verifies the local transfer
conditions.  The endpoint \(v=\infty\) is never used in the differentiated
transfer.  In the lower bound, if a selected context has \(c=0\), Blackwell
monotonicity compares that experiment with the finite intermediate
degradation used in the proof of Theorem~\ref{thm:joint-law}.  The theorem
then gives the final assertion.
\end{proof}

\begin{theorem}[Coupled baseline at a prespecified radius]
\label{thm:qubit-baseline}
Let the disclosed radial-qubit directions range over an equatorial projective
circle, let the target be the binary entropy
\[
 s(t)=-\frac{1+t}{2}\log\frac{1+t}{2}
      -\frac{1-t}{2}\log\frac{1-t}{2},
\]
and let \(t_0\in[\tau,1-\tau]\) for fixed \(0<\tau<1/2\).
For every integer sequence \(k_n\geq2\), with \(H_n=n^{1/6}\),
\begin{equation}
 F_{n,k_n}(H_n)
 =\operatorname{artanh}^2(t_0)
    \tan^2\!\left(\frac{\pi}{2k_n}\right)
  +O_\tau(n^{-1/3}),
 \label{eq:qubit-baseline}
\end{equation}
uniformly in \(t_0\).  Consequently, if \(k_n\to\infty\) and
\(k_n=o(n^{1/6})\), then
\[
 \frac{F_{n,k_n}(n^{1/6})}{A_{k_n}}\longrightarrow1.
\]
\end{theorem}

The expansion \eqref{eq:qubit-baseline} is proved in Section S7 of the
Supplementary Material by coupling the menu and
oracle calculations before subtraction.  The key uniformity is
\(\cos\{\pi/(2k)\}\geq1/\sqrt2\): all moment, clipping, and van Trees
constants are therefore independent of \(k\).  The exponent \(1/6\) is a
proved sufficient range for this prespecified radius, not a lower bound on
what any method can achieve.  Corollary~\ref{cor:qubit}, by contrast, has no
sample--menu coupling because its localization radius is allowed to grow more
slowly.  The fixed-action estimators and the van Trees argument are standard,
and the fixed-menu qubit geometry is established above.  Coupling the menu and
oracle risks uniformly in \(k\) gives the triangular \((n,k_n)\) statement.

\subsection{Finite-sample binary diagnostics}
\label{sec:binary-numerics}

We numerically examine the identity-target binary experiment that underlies
the differentiated transfer theorem.  For
\(c_k=\cos\{\pi/(2k)\}\), put
\begin{align*}
 \mathcal R_{n,H}(c;t_0)
 &=\inf_T\sup_{|h|\leq H}
   \E_{t_0+h/\sqrt n,c}(T-h)^2,\\
 D_{n,k}(H;t_0)
 &=\mathcal R_{n,H}(c_k;t_0)-\mathcal R_{n,H}(1;t_0),
 \qquad A_k=c_k^{-2}-1.
\end{align*}
The sufficient count has the exact binomial law
\(J\sim\operatorname{Bin}\{n,p_h\}\), where
\(p_h=\{1+c(t_0+h/\sqrt n)\}/2\).  We therefore use deterministic
likelihood calculations rather than Monte Carlo.

For each value of \(c\), a prior supported on an optimization grid defines an
exact Bayes lower bound for the continuous-parameter minimax risk.  We evaluate
this Bayes risk in standard floating-point arithmetic and denote the reported
value by \(L_c\).  Its Bayes rule \(d_c(J)\) is evaluated on a separate
validation mesh
\(\mathcal G\).  If \(\Delta d_j=d_{j+1}-d_j\),
\(\Delta(d^2)_j=d_{j+1}^2-d_j^2\), and
\(\underline v_p=\inf_{|h|\leq H}p_h(1-p_h)\), the analytic derivative bound
used to extend the mesh maximum to the whole interval is
\begin{equation}
 \Lambda_c=\min\left\{
  \frac{2cH^2}{\sqrt{\underline v_p}}+4H,\,
  \frac{c\sqrt n}{2}
    \bigl(\|\Delta(d_c^2)\|_\infty
          +2H\|\Delta d_c\|_\infty\bigr)+4H
 \right\}.
 \label{eq:numerical-lipschitz}
\end{equation}
The two terms in \eqref{eq:numerical-lipschitz} follow, respectively, from a
score/Cauchy--Schwarz bound and the Bernstein derivative identity for
binomial expectations.
In exact arithmetic, the mesh maximum plus \(\Lambda_c\) times half the mesh
spacing gives an upper bound for the risk of this Bayes rule.  Our
implementation evaluates this quantity in standard floating-point arithmetic,
adds a reported guard, and denotes the resulting value by \(U_c\).  We
summarize the menu--oracle comparison by the guarded diagnostic interval
\begin{equation}
 \mathcal I_{n,k}(H;t_0)
 = [L_{c_k}-U_1,\,U_{c_k}-L_1].
 \label{eq:numerical-interval}
\end{equation}
Because the floating-point operations are not outward rounded,
\(\mathcal I_{n,k}(H;t_0)\) is not a certified enclosure.  The analytic
allowance controls continuum discretization, while the displayed interval is
used only as a deterministic numerical diagnostic.

Table~\ref{tab:binary-diagnostics} reports representative normalized versions
of \eqref{eq:numerical-interval} from the prespecified grid.  The optimization
and validation meshes
contain 401 and 100001 points for the explicit-radius rows and 801 and 200001
points for the fixed-radius sensitivity rows.  All 42 completed cases, the
source hash, continuum allowances, optimization gaps, and checksums are in
the accompanying source archive.  At \((n,k,t_0)=(4096,16,0.5)\), doubling
both meshes narrows the guarded diagnostic interval from \([0.3552,0.4119]\) to
\([0.3669,0.4002]\), nested inside the original interval.

\begin{table}[t]
\caption{Guarded deterministic finite-sample diagnostic intervals for
\(D_{n,k}(H;t_0)/A_k\).  The first three rows use the explicit admissible
radius in \eqref{eq:joint-radius}.  The last three hold
\((t_0,k)=(0.5,4)\) fixed and vary \(H\).}
\label{tab:binary-diagnostics}
\centering
\small
\begin{tabular}{lrrrrc}
\hline
Radius & \(t_0\) & \(n\) & \(k\) & \(H\) & Interval \\
\hline
Explicit & 0.25 & 16384 & 16 & 1.247 & [0.343, 0.402] \\
Explicit & 0.50 & 16384 & 16 & 1.247 & [0.359, 0.420] \\
Explicit & 0.75 & 16384 & 16 & 1.247 & [0.448, 0.511] \\
Fixed   & 0.50 & 16384 &  4 & 2.000 & [0.533, 0.540] \\
Fixed   & 0.50 & 16384 &  4 & 4.000 & [0.717, 0.753] \\
Fixed   & 0.50 & 16384 &  4 & 8.000 & [0.760, 0.983] \\
\hline
\end{tabular}
\end{table}

These finite-sample values do not estimate the limiting constant by
extrapolation.  At the displayed explicit radius, \(H\) is only 1.247; for
fixed \(k\), it is eventually capped by \(A_k^{-1/20}\), whereas the theorem
follows sequences with both \(k_n\to\infty\) and \(H_n\to\infty\).  The
fixed-radius rows show the predicted movement toward one as the local
interval expands, while the wider \(H=8\) interval records the analytic
continuum allowance.  These finite-sample diagnostics are not used in the
proof of the asymptotic theorem.

The inverse-information coordinate is unbounded for an action orthogonal to
the disclosed direction, but the preceding proof uses posterior transfer only
on a fixed neighborhood of the oracle.

\section{Discussion}
\label{sec:discussion}

The pre-disclosure constraint separates the sample and menu resources.  The
sample size determines estimation precision after an experiment has been
selected, whereas the menu size determines how closely the installed catalog
approximates the context-specific oracle.  Under a greatest-element Blackwell
order, Theorem~\ref{thm:collapse} reduces every adaptive menu policy to
repeated use of one installed experiment.  The fixed-menu effect is then the
inverse-information distortion \(D\), and its optimized value is \(A_k\).

The fixed-menu limit alone does not determine a joint limit when \(A_{k_n}\)
vanishes.  The differentiated argument compares Bayes-risk increments along
a Blackwell chain rather than approximating the two risk levels separately.
The Gaussian posterior-variance certificate determines the limiting
increment, and all-prior transfer carries it to the original experiment.
Theorem~\ref{thm:primitive-transfer} derives this transfer from likelihood and
generator conditions that do not depend on a prior.  Corollary
\ref{cor:primitive-joint} records the resulting direct route from these
conditions to the joint sample-menu law.  Proposition~\ref{prop:lan-witness}
shows why pointwise Gaussian convergence is not an adequate substitute.

The geometric and statistical parts of the argument remain separate.  Metric
covering determines the behavior of \(A_k\); differentiated transfer shows
that the finite-neighborhood minimax excess has the same first-order value.
On an Ahlfors-regular oracle image this gives an original-scale excess mean
squared error of order \(n^{-1}k_n^{-2/r}\).  Corollary
\ref{cor:poisson-sensor-menu} obtains this rate for a calibrated Poisson sensor
catalog with an exact squared-distance frontier.  Corollary~\ref{cor:qubit}
gives the corresponding projective geometry for radial-qubit measurements.
Theorem~\ref{thm:qubit-baseline} provides a separate coupled calculation at a
prespecified localization radius.

Several restrictions are material.  A menu need not collapse if its
experiments are incomparable, as Example~\ref{ex:incomparable} shows.  The
one-dimensional chain also excludes vector parameters for which optimal
designs mix complementary actions.  The task quotient requires the oracle
sets to be fibers of one common map; genuinely set-valued oracle families may
lead to a different covering problem.  Finally, the primitive transfer theorem
requires a parameter-free degradation semigroup, a multiplicative local
likelihood approximation, and calibrated edge scores.  These conditions have
been verified here in four ordered families.  Extending the result to a general
DQM family would require control stronger than ordinary LAN and need not be
possible without an ordered degradation structure.

For partially ordered experiments, the natural design object is a menu of
allocations rather than a menu of single actions, and the induced distortion
will generally be matrix-valued before scalarization by the target.  Within
the ordered setting, sharper posterior-transfer bounds may permit faster
localization radii.  A second direction is to determine sharp covering
constants for broader oracle images under the intrinsic metric defined by the
Fisher-information Hessian.

\begin{appendix}

\section{Gaussian least-favorable posterior variance}
\label{app:gaussian-pvc}

\begin{proof}[Proof of Lemma~\ref{lem:gaussian-reference}]
Only \eqref{eq:gaussian-pvc} requires proof beyond the facts stated in the
text.  Gaussian scale equivariance gives
\[
 R_H^G(v)=v M(H/\sqrt v),\qquad M(a)=R_a^G(1).
\]
The function \(M\) is nondecreasing and bounded by one.  This gives the upper
bound in \eqref{eq:G-Lipschitz}; the lower bound follows by adding independent
Gaussian noise.  The cosine-squared prior used in the text gives
\eqref{eq:G-minimax-limit}.

For a prior \(\pi_H\) satisfying \eqref{eq:near-lf}, let
\(L_H=C(1+H^6)\) be a uniform curvature bound.  For \(s>0\),
\eqref{eq:G-Lipschitz} and least favorability imply
\[
 r_{v+s}^G(\pi_H)-r_v^G(\pi_H)\leq s+\eps_H.
\]
Taylor's theorem therefore gives
\[
 \partial_vr_v^G(\pi_H)
 \leq1+\eps_H/s+L_Hs/2.
\]
When \(\eps_H>0\), take \(s=(2\eps_H/L_H)^{1/2}\); the case
\(\eps_H=0\) follows by a limiting choice of \(s\).  Since
\(H^6\eps_H\to0\),
\[
 \limsup_{H\to\infty}\partial_vr_v^G(\pi_H)\leq1.
\]
On the other hand, \eqref{eq:G-minimax-limit} and
\eqref{eq:near-lf} imply
\[
 \E_{\pi_H}\{w_v^G(Y)/v\}=r_v^G(\pi_H)/v\longrightarrow1.
\]
By \eqref{eq:G-mmse-derivative}, the second moment of
\(w_v^G(Y)/v\) is \(\partial_vr_v^G(\pi_H)\).  Jensen's inequality gives the
matching lower limit.  Expanding the centered square proves
\eqref{eq:gaussian-pvc}.
\end{proof}

\section{A uniform smooth-target risk comparison}
\label{app:smooth-target}

\begin{lemma}
\label{lem:smooth-risk}
Suppose
\[
 g_n(h)=b h+\rho_n(h),\qquad
 \sup_{|h|\leq H}|\rho_n(h)|\leq C H^2/\sqrt n.
\]
For every statistical experiment and every prior on \([-H,H]\),
\[
 \left|r^{g_n}(\pi)-b^2r^h(\pi)\right|
 \leq C'\{H^3/\sqrt n+H^4/n\}.
\]
The same bound holds for the corresponding minimax risks.
\end{lemma}

\begin{proof}
The Bayes risks are expected posterior variances.  Conditional on the
observation,
\[
 \Var(g_n(h))=b^2\Var(h)+2b\Cov\{h,\rho_n(h)\}
              +\Var\{\rho_n(h)\}.
\]
The conditional standard deviations of \(h\) and \(\rho_n(h)\) are bounded
by \(H\) and \(CH^2/\sqrt n\), respectively.  Cauchy--Schwarz gives the
Bayes-risk bound.  Taking the supremum over priors gives the minimax bound
whenever Bayes--minimax duality holds; the direct estimator comparison gives
the same conclusion without invoking an attaining prior.
\end{proof}

\end{appendix}

\begin{supplement}
\stitle{Supplement to the article on pre-disclosure experiment menus}
\sdescription{The supplement proves the posterior Dirichlet identity and the
uniform all-prior transfer for binary attenuation and Poisson thinning. It
also proves the primitive likelihood-generator theorem, verifies its
conditions for binary, Poisson, and negative-binomial experiments, and gives
the smooth-target transfer and the coupled radial-qubit baseline.}
\end{supplement}

\bibliographystyle{imsart-nameyear}
\bibliography{refs}

\begin{thebibliography}{30}

\bibitem[\protect\citeauthoryear{Adusumilli}{2025}]{Adusumilli2025}
\begin{barticle}[author]
\bauthor{\bsnm{Adusumilli},~\bfnm{Karun}\binits{K.}}
(\byear{2025}).
\btitle{Risk and Optimal Policies in Bandit Experiments}.
\bjournal{Econometrica}
\bvolume{93}
\bpages{1003--1029}.
\bdoi{10.3982/ECTA21075}
\end{barticle}
\endbibitem

\bibitem[\protect\citeauthoryear{Adusumilli}{2026}]{Adusumilli2026}
\begin{bmisc}[author]
\bauthor{\bsnm{Adusumilli},~\bfnm{Karun}\binits{K.}}
(\byear{2026}).
\btitle{Continuous Time Asymptotic Representations for Adaptive Experiments}.
\bnote{\href{https://arxiv.org/abs/2601.00739}{arXiv:2601.00739}}.
\end{bmisc}
\endbibitem

\bibitem[\protect\citeauthoryear{Armstrong}{2022}]{Armstrong2022}
\begin{bmisc}[author]
\bauthor{\bsnm{Armstrong},~\bfnm{Timothy~B.}\binits{T.~B.}}
(\byear{2022}).
\btitle{Asymptotic Efficiency Bounds for a Class of Experimental Designs}.
\bnote{\href{https://arxiv.org/abs/2205.02726}{arXiv:2205.02726}}.
\end{bmisc}
\endbibitem

\bibitem[\protect\citeauthoryear{Ballester, Wehner and
  Winter}{2008}]{BallesterEtAl2008}
\begin{barticle}[author]
\bauthor{\bsnm{Ballester},~\bfnm{Manuel~A.}\binits{M.~A.}},
  \bauthor{\bsnm{Wehner},~\bfnm{Stephanie}\binits{S.}} \AND
  \bauthor{\bsnm{Winter},~\bfnm{Andreas}\binits{A.}}
(\byear{2008}).
\btitle{State Discrimination with Post-Measurement Information}.
\bjournal{IEEE Transactions on Information Theory}
\bvolume{54}
\bpages{4183--4198}.
\bdoi{10.1109/TIT.2008.928276}
\end{barticle}
\endbibitem

\bibitem[\protect\citeauthoryear{Bergemann, Bonatti and
  Smolin}{2018}]{BergemannBonattiSmolin2018}
\begin{barticle}[author]
\bauthor{\bsnm{Bergemann},~\bfnm{Dirk}\binits{D.}},
  \bauthor{\bsnm{Bonatti},~\bfnm{Alessandro}\binits{A.}} \AND
  \bauthor{\bsnm{Smolin},~\bfnm{Alex}\binits{A.}}
(\byear{2018}).
\btitle{The Design and Price of Information}.
\bjournal{American Economic Review}
\bvolume{108}
\bpages{1--48}.
\bdoi{10.1257/aer.20161079}
\end{barticle}
\endbibitem

\bibitem[\protect\citeauthoryear{Bergemann, Yeh and
  Zhang}{2021}]{BergemannYehZhang2021}
\begin{barticle}[author]
\bauthor{\bsnm{Bergemann},~\bfnm{Dirk}\binits{D.}},
  \bauthor{\bsnm{Yeh},~\bfnm{Edmund}\binits{E.}} \AND
  \bauthor{\bsnm{Zhang},~\bfnm{Jinkun}\binits{J.}}
(\byear{2021}).
\btitle{Nonlinear Pricing with Finite Information}.
\bjournal{Games and Economic Behavior}
\bvolume{130}
\bpages{62--84}.
\bdoi{10.1016/j.geb.2021.08.004}
\end{barticle}
\endbibitem

\bibitem[\protect\citeauthoryear{Bertsimas and
  Caramanis}{2010}]{BertsimasCaramanis2010}
\begin{barticle}[author]
\bauthor{\bsnm{Bertsimas},~\bfnm{Dimitris}\binits{D.}} \AND
  \bauthor{\bsnm{Caramanis},~\bfnm{Constantine}\binits{C.}}
(\byear{2010}).
\btitle{Finite Adaptability in Multistage Linear Optimization}.
\bjournal{IEEE Transactions on Automatic Control}
\bvolume{55}
\bpages{2751--2766}.
\bdoi{10.1109/TAC.2010.2049764}
\end{barticle}
\endbibitem

\bibitem[\protect\citeauthoryear{Bertsimas, Goyal and
  Sun}{2011}]{BertsimasGoyalSun2011}
\begin{barticle}[author]
\bauthor{\bsnm{Bertsimas},~\bfnm{Dimitris}\binits{D.}},
  \bauthor{\bsnm{Goyal},~\bfnm{Vineet}\binits{V.}} \AND
  \bauthor{\bsnm{Sun},~\bfnm{Xu~Andy}\binits{X.~A.}}
(\byear{2011}).
\btitle{A Geometric Characterization of the Power of Finite Adaptability in
  Multistage Stochastic and Adaptive Optimization}.
\bjournal{Mathematics of Operations Research}
\bvolume{36}
\bpages{24--54}.
\bdoi{10.1287/moor.1110.0482}
\end{barticle}
\endbibitem

\bibitem[\protect\citeauthoryear{Blackwell}{1953}]{Blackwell1953}
\begin{barticle}[author]
\bauthor{\bsnm{Blackwell},~\bfnm{David}\binits{D.}}
(\byear{1953}).
\btitle{Equivalent Comparisons of Experiments}.
\bjournal{The Annals of Mathematical Statistics}
\bvolume{24}
\bpages{265--272}.
\bdoi{10.1214/aoms/1177729032}
\end{barticle}
\endbibitem

\bibitem[\protect\citeauthoryear{Braess and Dette}{2007}]{BraessDette2007}
\begin{barticle}[author]
\bauthor{\bsnm{Braess},~\bfnm{Dietrich}\binits{D.}} \AND
  \bauthor{\bsnm{Dette},~\bfnm{Holger}\binits{H.}}
(\byear{2007}).
\btitle{On the Number of Support Points of Maximin and {Bayesian} Optimal
  Designs}.
\bjournal{The Annals of Statistics}
\bvolume{35}
\bpages{772--792}.
\bdoi{10.1214/009053606000001307}
\end{barticle}
\endbibitem

\bibitem[\protect\citeauthoryear{Carmeli, Heinosaari and
  Toigo}{2018}]{CarmeliEtAl2018}
\begin{barticle}[author]
\bauthor{\bsnm{Carmeli},~\bfnm{Claudio}\binits{C.}},
  \bauthor{\bsnm{Heinosaari},~\bfnm{Teiko}\binits{T.}} \AND
  \bauthor{\bsnm{Toigo},~\bfnm{Alessandro}\binits{A.}}
(\byear{2018}).
\btitle{State Discrimination with Post-Measurement Information and
  Incompatibility of Quantum Measurements}.
\bjournal{Physical Review A}
\bvolume{98}
\bpages{012126}.
\bdoi{10.1103/PhysRevA.98.012126}
\end{barticle}
\endbibitem

\bibitem[\protect\citeauthoryear{Carter}{2002}]{Carter2002}
\begin{barticle}[author]
\bauthor{\bsnm{Carter},~\bfnm{Andrew~V.}\binits{A.~V.}}
(\byear{2002}).
\btitle{Deficiency Distance between Multinomial and Multivariate Normal
  Experiments}.
\bjournal{The Annals of Statistics}
\bvolume{30}
\bpages{708--730}.
\bdoi{10.1214/aos/1028674839}
\end{barticle}
\endbibitem

\bibitem[\protect\citeauthoryear{DeGroot}{1962}]{DeGroot1962}
\begin{barticle}[author]
\bauthor{\bsnm{DeGroot},~\bfnm{Morris~H.}\binits{M.~H.}}
(\byear{1962}).
\btitle{Uncertainty, Information, and Sequential Experiments}.
\bjournal{The Annals of Mathematical Statistics}
\bvolume{33}
\bpages{404--419}.
\bdoi{10.1214/aoms/1177704567}
\end{barticle}
\endbibitem

\bibitem[\protect\citeauthoryear{Farias and
  Brossier}{2014}]{FariasBrossier2014}
\begin{barticle}[author]
\bauthor{\bsnm{Farias},~\bfnm{Rodrigo~Cabral}\binits{R.~C.}} \AND
  \bauthor{\bsnm{Brossier},~\bfnm{Jean-Marc}\binits{J.-M.}}
(\byear{2014}).
\btitle{Scalar Quantization for Estimation: From an Asymptotic Design to a
  Practical Solution}.
\bjournal{IEEE Transactions on Signal Processing}
\bvolume{62}
\bpages{2860--2870}.
\bdoi{10.1109/TSP.2014.2318140}
\end{barticle}
\endbibitem

\bibitem[\protect\citeauthoryear{Gill and Levit}{1995}]{GillLevit1995}
\begin{barticle}[author]
\bauthor{\bsnm{Gill},~\bfnm{Richard~D.}\binits{R.~D.}} \AND
  \bauthor{\bsnm{Levit},~\bfnm{Boris~Y.}\binits{B.~Y.}}
(\byear{1995}).
\btitle{Applications of the van {Trees} Inequality: A {Bayesian}
  {Cram\'er--Rao} Bound}.
\bjournal{Bernoulli}
\bvolume{1}
\bpages{59--79}.
\bdoi{10.2307/3318681}
\end{barticle}
\endbibitem

\bibitem[\protect\citeauthoryear{Gopal and Wehner}{2010}]{GopalWehner2010}
\begin{barticle}[author]
\bauthor{\bsnm{Gopal},~\bfnm{Deepthi}\binits{D.}} \AND
  \bauthor{\bsnm{Wehner},~\bfnm{Stephanie}\binits{S.}}
(\byear{2010}).
\btitle{Using Post-Measurement Information in State Discrimination}.
\bjournal{Physical Review A}
\bvolume{82}
\bpages{022326}.
\bdoi{10.1103/PhysRevA.82.022326}
\end{barticle}
\endbibitem

\bibitem[\protect\citeauthoryear{Graf and Luschgy}{2000}]{GrafLuschgy2000}
\begin{bbook}[author]
\bauthor{\bsnm{Graf},~\bfnm{Siegfried}\binits{S.}} \AND
  \bauthor{\bsnm{Luschgy},~\bfnm{Harald}\binits{H.}}
(\byear{2000}).
\btitle{Foundations of Quantization for Probability Distributions}.
\bseries{Lecture Notes in Mathematics}
\bvolume{1730}.
\bpublisher{Springer}, \baddress{Berlin}.
\bdoi{10.1007/BFb0103945}
\end{bbook}
\endbibitem

\bibitem[\protect\citeauthoryear{Gruber}{1998}]{Gruber1998}
\begin{barticle}[author]
\bauthor{\bsnm{Gruber},~\bfnm{Peter~M.}\binits{P.~M.}}
(\byear{1998}).
\btitle{Asymptotic Estimates for Best and Stepwise Approximation of Convex
  Bodies {IV}}.
\bjournal{Forum Mathematicum}
\bvolume{10}
\bpages{665--686}.
\bdoi{10.1515/form.10.6.665}
\end{barticle}
\endbibitem

\bibitem[\protect\citeauthoryear{Guo, Shamai and
  Verd{\'u}}{2005}]{GuoShamaiVerdu2005}
\begin{barticle}[author]
\bauthor{\bsnm{Guo},~\bfnm{Dongning}\binits{D.}},
  \bauthor{\bsnm{Shamai},~\bfnm{Shlomo}\binits{S.}} \AND
  \bauthor{\bsnm{Verd{\'u}},~\bfnm{Sergio}\binits{S.}}
(\byear{2005}).
\btitle{Mutual Information and Minimum Mean-Square Error in Gaussian Channels}.
\bjournal{IEEE Transactions on Information Theory}
\bvolume{51}
\bpages{1261--1282}.
\bdoi{10.1109/TIT.2005.844072}
\end{barticle}
\endbibitem

\bibitem[\protect\citeauthoryear{Guo et~al.}{2011}]{GuoWuShamaiVerdu2011}
\begin{barticle}[author]
\bauthor{\bsnm{Guo},~\bfnm{Dongning}\binits{D.}},
  \bauthor{\bsnm{Wu},~\bfnm{Yihong}\binits{Y.}},
  \bauthor{\bsnm{Shamai},~\bfnm{Shlomo}\binits{S.}} \AND
  \bauthor{\bsnm{Verd{\'u}},~\bfnm{Sergio}\binits{S.}}
(\byear{2011}).
\btitle{Estimation in Gaussian Noise: Properties of the Minimum Mean-Square
  Error}.
\bjournal{IEEE Transactions on Information Theory}
\bvolume{57}
\bpages{2371--2385}.
\bdoi{10.1109/TIT.2011.2111010}
\end{barticle}
\endbibitem

\bibitem[\protect\citeauthoryear{Hancart}{2026}]{Hancart2026}
\begin{bunpublished}[author]
\bauthor{\bsnm{Hancart},~\bfnm{Nathan}\binits{N.}}
(\byear{2026}).
\btitle{The Optimal Menu of Tests}.
\bnote{Working paper,
  \href{https://nathanhancart.com/documents/optimal_menu.pdf} {available
  online}}.
\end{bunpublished}
\endbibitem

\bibitem[\protect\citeauthoryear{Hirano and Porter}{2023}]{HiranoPorter2023}
\begin{bmisc}[author]
\bauthor{\bsnm{Hirano},~\bfnm{Keisuke}\binits{K.}} \AND
  \bauthor{\bsnm{Porter},~\bfnm{Jack~R.}\binits{J.~R.}}
(\byear{2023}).
\btitle{Asymptotic Representations for Sequential Decisions, Adaptive
  Experiments, and Batched Bandits}.
\bnote{\href{https://arxiv.org/abs/2302.03117}{arXiv:2302.03117}}.
\end{bmisc}
\endbibitem

\bibitem[\protect\citeauthoryear{Kedad-Sidhoum, Medvedev and
  Meunier}{2023}]{KedadSidhoumMedvedevMeunier2023}
\begin{bmisc}[author]
\bauthor{\bsnm{Kedad-Sidhoum},~\bfnm{Safia}\binits{S.}},
  \bauthor{\bsnm{Medvedev},~\bfnm{Anton}\binits{A.}} \AND
  \bauthor{\bsnm{Meunier},~\bfnm{Fr{\'e}d{\'e}ric}\binits{F.}}
(\byear{2023}).
\btitle{Finite Adaptability in Two-Stage Robust Optimization: Asymptotic
  Optimality and Tractability}.
\bnote{\href{https://arxiv.org/abs/2305.05399}{arXiv:2305.05399}}.
\end{bmisc}
\endbibitem

\bibitem[\protect\citeauthoryear{Kurtz}{2026}]{Kurtz2026}
\begin{barticle}[author]
\bauthor{\bsnm{Kurtz},~\bfnm{Jannis}\binits{J.}}
(\byear{2026}).
\btitle{Bounding the Optimal Number of Policies for Robust {$K$}-Adaptability}.
\bjournal{Mathematical Programming}.
\bnote{Published online 29 January 2026}.
\bdoi{10.1007/s10107-026-02329-1}
\end{barticle}
\endbibitem

\bibitem[\protect\citeauthoryear{Li and Zhao}{2025}]{LiZhao2025}
\begin{barticle}[author]
\bauthor{\bsnm{Li},~\bfnm{Xiaoou}\binits{X.}} \AND
  \bauthor{\bsnm{Zhao},~\bfnm{Hongru}\binits{H.}}
(\byear{2025}).
\btitle{Globally-Optimal Greedy Active Sequential Estimation}.
\bjournal{IEEE Transactions on Information Theory}
\bvolume{71}
\bpages{3871--3924}.
\bdoi{10.1109/TIT.2025.3551621}
\end{barticle}
\endbibitem

\bibitem[\protect\citeauthoryear{Rezaei, Wei and Han}{2026}]{RezaeiWeiHan2026}
\begin{bmisc}[author]
\bauthor{\bsnm{Rezaei},~\bfnm{Zolykha}\binits{Z.}},
  \bauthor{\bsnm{Wei},~\bfnm{Ningji}\binits{N.}} \AND
  \bauthor{\bsnm{Han},~\bfnm{Eojin}\binits{E.}}
(\byear{2026}).
\btitle{Scalable Finite Adaptability via Polyhedral Partition and Learning}.
\bnote{\href{https://arxiv.org/abs/2606.06927}{arXiv:2606.06927}}.
\end{bmisc}
\endbibitem

\bibitem[\protect\citeauthoryear{Subramanyam, Gounaris and
  Wiesemann}{2020}]{SubramanyamEtAl2020}
\begin{barticle}[author]
\bauthor{\bsnm{Subramanyam},~\bfnm{Anirudh}\binits{A.}},
  \bauthor{\bsnm{Gounaris},~\bfnm{Chrysanthos~E.}\binits{C.~E.}} \AND
  \bauthor{\bsnm{Wiesemann},~\bfnm{Wolfram}\binits{W.}}
(\byear{2020}).
\btitle{{$K$}-Adaptability in Two-Stage Mixed-Integer Robust Optimization}.
\bjournal{Mathematical Programming Computation}
\bvolume{12}
\bpages{193--224}.
\bdoi{10.1007/s12532-019-00174-2}
\end{barticle}
\endbibitem

\bibitem[\protect\citeauthoryear{Takanashi and
  McAlinn}{2026}]{TakanashiMcAlinn2026}
\begin{bmisc}[author]
\bauthor{\bsnm{Takanashi},~\bfnm{K{\=o}saku}\binits{K.}} \AND
  \bauthor{\bsnm{McAlinn},~\bfnm{Kenichiro}\binits{K.}}
(\byear{2026}).
\btitle{An Entropy--Energy Identity for Predictive {Kullback--Leibler} Regret
  in Infinitely Divisible Location Models}.
\bnote{\href{https://arxiv.org/abs/2605.27253}{arXiv:2605.27253}}.
\end{bmisc}
\endbibitem

\bibitem[\protect\citeauthoryear{van~der Vaart}{1998}]{vanDerVaart1998}
\begin{bbook}[author]
\bauthor{\bparticle{van~der} \bsnm{Vaart},~\bfnm{Aad~W.}\binits{A.~W.}}
(\byear{1998}).
\btitle{Asymptotic Statistics}.
\bpublisher{Cambridge University Press}, \baddress{Cambridge}.
\bdoi{10.1017/CBO9780511802256}
\end{bbook}
\endbibitem

\bibitem[\protect\citeauthoryear{Zhu and Joe}{2010}]{ZhuJoe2010}
\begin{barticle}[author]
\bauthor{\bsnm{Zhu},~\bfnm{Rong}\binits{R.}} \AND
  \bauthor{\bsnm{Joe},~\bfnm{Harry}\binits{H.}}
(\byear{2010}).
\btitle{Negative Binomial Time Series Models Based on Expectation Thinning
  Operators}.
\bjournal{Journal of Statistical Planning and Inference}
\bvolume{140}
\bpages{1874--1888}.
\bdoi{10.1016/j.jspi.2010.01.031}
\end{barticle}
\endbibitem

\end{thebibliography}


\begin{thebibliography}{3}

\bibitem[\protect\citeauthoryear{Carter}{2002}]{Carter2002}
\begin{barticle}[author]
\bauthor{\bsnm{Carter},~\bfnm{Andrew~V.}\binits{A.~V.}}
(\byear{2002}).
\btitle{Deficiency Distance between Multinomial and Multivariate Normal
  Experiments}.
\bjournal{The Annals of Statistics}
\bvolume{30}
\bpages{708--730}.
\bdoi{10.1214/aos/1028674839}
\end{barticle}
\endbibitem

\bibitem[\protect\citeauthoryear{Gill and Levit}{1995}]{GillLevit1995}
\begin{barticle}[author]
\bauthor{\bsnm{Gill},~\bfnm{Richard~D.}\binits{R.~D.}} \AND
  \bauthor{\bsnm{Levit},~\bfnm{Boris~Y.}\binits{B.~Y.}}
(\byear{1995}).
\btitle{Applications of the van {Trees} Inequality: A {Bayesian}
  {Cram\'er--Rao} Bound}.
\bjournal{Bernoulli}
\bvolume{1}
\bpages{59--79}.
\bdoi{10.2307/3318681}
\end{barticle}
\endbibitem

\bibitem[\protect\citeauthoryear{Zhu and Joe}{2010}]{ZhuJoe2010}
\begin{barticle}[author]
\bauthor{\bsnm{Zhu},~\bfnm{Rong}\binits{R.}} \AND
  \bauthor{\bsnm{Joe},~\bfnm{Harry}\binits{H.}}
(\byear{2010}).
\btitle{Negative Binomial Time Series Models Based on Expectation Thinning
  Operators}.
\bjournal{Journal of Statistical Planning and Inference}
\bvolume{140}
\bpages{1874--1888}.
\bdoi{10.1016/j.jspi.2010.01.031}
\end{barticle}
\endbibitem

\end{thebibliography}

\end{document}


\begin{frontmatter}
\pdfsubject{Supplement to a preprint manuscript}
\title{Supplement to Pre-Disclosure Experiment Menus:
Oracle-Relative Risk and Joint Sample--Menu Asymptotics}
\runtitle{Supplement to Pre-Disclosure Experiment Menus}

\begin{aug}
\author[A]{\fnms{Xinyu}~\snm{Song}\ead[label=e1]{song.xinyu@mail.shufe.edu.cn}}
\address[A]{School of Statistics and Data Science, Shanghai University of Finance and Economics\printead[presep={,\ }]{e1}}
\end{aug}

\begin{abstract}
This supplement proves the differentiated all-prior Gaussian transfer used in
the main paper.  It first proves a posterior Dirichlet identity for Markov
degradation semigroups and exact derivative formulas for binary attenuation
and Poisson thinning.  It then derives the all-prior risk and derivative
comparisons from prior-free local-likelihood, quadrature, edge-score, and
generator-calibration conditions.  Negative-binomial thinning verifies these
primitive conditions in an overdispersed count model.  The final sections give
the smooth-target transfer and the coupled fixed-radius radial-qubit baseline.
\end{abstract}
\end{frontmatter}

\section{Preliminaries}
\label{sec:s-prelim}

For an experiment indexed by an inverse-information coordinate \(v\), a prior
\(\pi\) on \([-H,H]\), and squared-error loss for a target \(g(h)\), write
\[
 r^g_{n,v}(\pi)=\E_\pi\Var_\pi\{g(h)\mid Y\}.
\]
The identity target is denoted by \(g(h)=h\).  The Gaussian comparison
experiment is \(Y=h+\sqrt v Z\), \(Z\sim N(0,1)\), and its posterior mean and
variance are written \(\mu_\pi(y)\) and \(w_\pi(y)\).  Its predictive density
is \(m_\pi(y)\).

For any continuous target \(g\) on \([-H,H]\), the
finite-parameter-interval minimax identity used below is
\begin{equation}
 R^g_{n,H}(v)=\sup_{\pi\in\cP([-H,H])}r^g_{n,v}(\pi).
 \label{eq:s-minimax}
\end{equation}
Clipping decisions to the compact range of \(g\) does not increase risk.  The
Gaussian, binary, Poisson, and negative-binomial experiments are dominated by
fixed measures,
their likelihoods are continuous in the local parameter in \(L^1\), and the
parameter interval is compact.  The compact statistical minimax theorem gives
\eqref{eq:s-minimax}.  Since Bayes risk is upper semicontinuous on the weakly
compact prior space, the supremum is attained.  These conclusions apply both
to the identity target and to every smooth local target used below.

\section{Posterior Dirichlet identity}
\label{sec:s-dirichlet}

We prove the posterior Dirichlet identity in the main paper.  Let
\(G=g(U)\), where \(U\) denotes the unknown quantity.  First suppose
the degradation semigroup has a finite or countable state space.  Let
\[
 \mu_s(x)=\Pp(X_s=x),\qquad
 \nu_s(x)=\E\{G\ind(X_s=x)\},\qquad
 m_s(x)=\nu_s(x)/\mu_s(x).
\]
At the differentiation point, assume \(\mu_s(z)>0\) whenever
\(\mu_s(x)q(x,z)\neq0\) for some \(x\).  Thus every posterior mean appearing
in a positive-flow term below is defined.  Both predictive measures evolve
under the same parameter-free forward equation,
\[
 \partial_s\mu_s(z)=\sum_x\mu_s(x)q(x,z),\qquad
 \partial_s\nu_s(z)=\sum_x\nu_s(x)q(x,z).
\]
The posterior explained variance is
\(\sum_z\nu_s(z)^2/\mu_s(z)\).  Absolute convergence and differentiation give
\begin{align*}
 \partial_s\sum_z\frac{\nu_s(z)^2}{\mu_s(z)}
 &=\sum_{x,z}\mu_s(x)q(x,z)
       \{2m_s(x)m_s(z)-m_s(z)^2\}\\
 &=-\sum_{x,z}\mu_s(x)q(x,z)
       \{m_s(z)-m_s(x)\}^2.
\end{align*}
The second line uses \(\sum_zq(x,z)=0\).
Bayes risk is the prior second moment of \(G\) minus posterior explained
variance, proving the jump identity.

For Brownian convolution, let \(\mu_s\) and \(\nu_s\) denote the predictive
density and the target-weighted predictive density.  They both solve the heat
equation \(\partial_sf=\frac12\Delta f\), and
\(m_s=\nu_s/\mu_s\).  Integration by parts gives
\[
 \partial_s\int\frac{\nu_s^2}{\mu_s}
 =-\int\mu_s\|\nabla m_s\|^2.
\]
Subtracting from the fixed prior second moment proves the heat identity.  The
required boundary terms vanish under the conditions in the proposition.

\section{Binary attenuation}
\label{sec:s-binary}

Fix \(t_0\) in a compact subset of \((-1,1)\).  For
\[
 x=c^{-2}\in K\Subset(t_0^2,\infty),\qquad
 t_n(h)=t_0+h/\sqrt n,\qquad |h|\leq H,
\]
observe \(n\) independent binary variables with
\[
 \Pp_h(Y_i=y)=\{1+yct_n(h)\}/2,\qquad y\in\{-1,1\}.
\]
The inverse Fisher information at \(t_0\) is
\[
 v=V_x:=x-t_0^2.
\]
Write \(r^B_{n,x}(\pi)\) for the Bayes risk for estimating \(h\).

\subsection{Exact derivative along the binary channel}

Let \(J_-\) be the number of plus observations among the first \(n-1\)
variables.  Given \(J_-=j\), let
\[
 \mu_j=\E_\pi(h\mid J_-=j),\qquad
 w_j=\Var_\pi(h\mid J_-=j).
\]

\begin{lemma}[Binary posterior derivative]
\label{lem:s-binary-derivative}
For every prior on \([-H,H]\),
\begin{equation}
 \partial_xr^B_{n,x}(\pi)
 =\E_{\pi,x}^{(n-1)}
 \left[
  \frac{w_{J_-}}
       {x-\{t_0+\mu_{J_-}/\sqrt n\}^2}
 \right]^2.
 \label{eq:s-binary-derivative}
\end{equation}
\end{lemma}

\begin{proof}
At an arbitrary starting value \(c_0>0\), parameterize an additional binary
symmetric channel by \(s\geq0\), with correlation
\(c_s=c_0e^{-2s}\) and \(x_s=c_s^{-2}=x_0e^{4s}\).  For a bounded target and a binary vector
\(Y\), let \(m(y)=\E(h\mid Y=y)\), and let \(Y^{(i)}\) be obtained by flipping
coordinate \(i\).  Differentiating the posterior explained variance when an
additional binary channel with flip probability \(q\) is applied, and pairing
the two orientations of each hypercube edge, gives
\begin{equation}
 \left.\frac{\dd}{\dd q}r_q(\pi)\right|_{q=0}
 =\sum_{i=1}^n\E\{m(Y)-m(Y^{(i)})\}^2.
 \label{eq:s-hypercube-dirichlet}
\end{equation}
Along the continuous-time channel, the same expression is
\(\partial_sr_s(\pi)\).

By exchangeability, the posterior mean depends only on the plus count.
Conditional on \(J_-=j\), write
\[
 p(h)=\{1+c(t_0+h/\sqrt n)\}/2,\qquad
 \bar p_j=\E_\pi\{p(h)\mid J_-=j\}.
\]
Appending a plus or a minus observation produces adjacent posterior means
\(m_{j+1}\) and \(m_j\), and the affine dependence of \(p(h)\) on \(h\) gives
\[
 m_{j+1}-m_j
 =\frac{\Cov_\pi\{h,p(h)\mid J_-=j\}}
        {\bar p_j(1-\bar p_j)}
 =\frac{cw_j}{2\sqrt n\,\bar p_j(1-\bar p_j)}.
\]
Since
\[
 4\bar p_j(1-\bar p_j)
 =1-c^2\{t_0+\mu_j/\sqrt n\}^2,
\]
the normalized adjacent slope
\[
 d_{x,j}=\frac{c\sqrt n}{2}(m_{j+1}-m_j)
\]
equals
\[
 d_{x,j}=\frac{w_j}{x-\{t_0+\mu_j/\sqrt n\}^2}.
\]
Equation \eqref{eq:s-hypercube-dirichlet}, \(dx/ds=4x\), and the \(j\)
downward and \(n-j\) upward edges show that
\[
 \partial_xr^B_{n,x}(\pi)
 =\E_{\pi,x}\left[
   \frac Jn d_{x,J-1}^2+
   \left(1-\frac Jn\right)d_{x,J}^2\right].
\]
The binomial recursion
\[
 M_n(j)(1-j/n)+M_n(j+1)(j+1)/n=M_{n-1}(j)
\]
combines the two orientations of each edge into the predictive law of
\(J_-\).  This proves \eqref{eq:s-binary-derivative}.
\end{proof}

\subsection{Uniform binary-to-Gaussian comparison}

Define the unbiased local statistic
\[
 T_{n,x}(J)=\frac{2J-n}{c\sqrt n}-\sqrt n\,t_0.
\]
At \(h=0\), its variance is \(V_x=x-t_0^2\).

\begin{lemma}[All-prior binary transfer]
\label{lem:s-binary-transfer}
There are constants \(C,c_0>0\), uniform for \(x\in K\), such that whenever
\(1\leq H=o(n^{1/6})\), for every prior on \([-H,H]\),
\begin{align}
 |r^B_{n,x}(\pi)-r^G_{V_x}(\pi)|
 &\leq C H^5n^{-1/2}+CH^2e^{-c_0H^2},                 \label{eq:s-binary-risk}\\
 \left|\partial_xr^B_{n,x}(\pi)
 -\E_\pi^G\{w_\pi(Y)/V_x\}^2\right|
 &\leq C H^7n^{-1/2}+CH^4e^{-c_0H^2}.                 \label{eq:s-binary-slope}
\end{align}
The centered comparison
\begin{align}
 &\left|\E_{\pi,x}^{(n-1)}
  \left[
   \frac{w_{J_-}}{x-\{t_0+\mu_{J_-}/\sqrt n\}^2}-1
  \right]^2
 -\E_\pi^G\{w_\pi(Y)/V_x-1\}^2\right| \notag\\
 &\hspace{25mm}\leq
 C H^7n^{-1/2}+CH^4e^{-c_0H^2}+CH^6n^{-1}
 \label{eq:s-binary-centered}
\end{align}
also holds.
\end{lemma}

\begin{proof}
Put
\[
 p_0=\{1+ct_0\}/2,\qquad p_h=p_0+\frac{ch}{2\sqrt n}.
\]
At a count for which \(T_{n,x}(J)=y\), a third-order Taylor expansion of the
exact binomial likelihood ratio gives
\begin{align}
 \log\frac{\Pp_h(J)}{\Pp_0(J)}
 &=\frac{hy}{V_x}-\frac{h^2}{2V_x}+R_{n,x}(h,y),
                                                        \label{eq:s-bin-lr}\\
 |R_{n,x}(h,y)|
 &\leq C\frac{|h|^2|y|+|h|^3}{\sqrt n}.                \notag
\end{align}
To see the coefficients, substitute
\(J=np_0+c\sqrt n\,y/2\) into
\[
 J\log(p_h/p_0)+(n-J)\log\{(1-p_h)/(1-p_0)\}.
\]
The compact restrictions on \(t_0\) and \(x\) bound all third derivatives.

The lattice of \(T_{n,x}\) has spacing \(a_n=2/(c\sqrt n)\).
Stirling's formula, uniformly for \(|y|=o(n^{1/6})\), gives
\begin{equation}
 \Pp_0\{T_{n,x}=y\}
 =a_n\phi_{V_x}(y)
 \left[1+O\{(1+|y|^3)/\sqrt n\}\right].
 \label{eq:s-bin-local-limit}
\end{equation}
This is the one-dimensional instance of the comparison used by
\citet{Carter2002}.  Combining \eqref{eq:s-bin-lr}--\eqref{eq:s-bin-local-limit}
with
\[
 \phi_{V_x}(y)\exp\{hy/V_x-h^2/(2V_x)\}
 =\phi_{V_x}(y-h)
\]
shows that, on the central lattice set
\(\mathcal C_n=\{y:|y|\leq C_1H\}\),
\begin{equation}
 e^{-\eps_n}a_n\phi_{V_x}(y-h)
 \leq \Pp_h\{T_{n,x}=y\}
 \leq e^{\eps_n}a_n\phi_{V_x}(y-h),
 \qquad
 \eps_n=CH^3/\sqrt n=o(1).
 \label{eq:s-bin-relative}
\end{equation}

Uniform binomial Bernstein bounds, and the corresponding Gaussian bound,
give
\begin{equation}
 \sup_{|h|\leq H}
 \Pp_h\{T_{n,x}\notin\mathcal C_n\}
 \leq Ce^{-c_0H^2}.
 \label{eq:s-bin-tail}
\end{equation}
Mixing \eqref{eq:s-bin-relative} over an arbitrary prior shows that its binary
predictive mass is within the same multiplicative factors of
\(a_nm_\pi(y)\).  After normalization, the binary posterior density relative
to the Gaussian posterior lies between \(e^{-2\eps_n}\) and
\(e^{2\eps_n}\).  Consequently, for a bounded function \(f\), the two
posterior expectations differ by at most
\((e^{2\eps_n}-1)\osc(f)\).  In particular, on \(\mathcal C_n\),
\begin{equation}
 |\mu_J-\mu_\pi(y)|\leq CH^4/\sqrt n,\qquad
 |w_J-w_\pi(y)|\leq CH^5/\sqrt n.
 \label{eq:s-bin-moments}
\end{equation}

It remains to compare the predictive sum with the Gaussian integral uniformly
over the prior.  Gaussian posterior differentiation gives
\[
 m_\pi'=m_\pi(\mu_\pi-y)/V_x,\qquad
 \mu_\pi'=w_\pi/V_x,\qquad
 w_\pi'=\kappa_{3,\pi}/V_x,
\]
where \(\kappa_{3,\pi}\) is the posterior third central moment.  Since
\(w_\pi\leq H^2\), \(|\kappa_{3,\pi}|\leq CH^3\), and
\(\int m_\pi|\mu_\pi-y|\leq C\), these identities yield the sufficient
total-variation bounds.  The last inequality follows from
\(\E|\mu_\pi(Y)-Y|\leq\E|h-Y|\) in the Gaussian experiment.  Thus
\begin{align}
 \int_\R |\partial_y(m_\pi w_\pi)|\,\dd y
 &\leq C(1+H^3),                                      \label{eq:s-bin-tv1}\\
 \int_\R \left|\partial_y\left[
 m_\pi\{w_\pi/V_x\}^2\right]\right|\,\dd y
 &\leq C(1+H^5),                                      \label{eq:s-bin-tv2}\\
 \int_\R \left|\partial_y\left[
 m_\pi\{w_\pi/V_x-1\}^2\right]\right|\,\dd y
 &\leq C(1+H^5).                                      \label{eq:s-bin-tv3}
\end{align}
For every absolutely continuous integrable \(f\), the rectangle-rule error
on a lattice of spacing \(a_n\) is at most
\(a_n\int|f'|\), apart from Gaussian tails beyond the finite binomial lattice,
which are exponentially small uniformly over the local prior mixture.

The binary Bayes risk is its predictive expectation of posterior variance.
Equations \eqref{eq:s-bin-relative}--\eqref{eq:s-bin-tv1} give the central
error \(O(H^5/\sqrt n)\), and \eqref{eq:s-bin-tail} gives the tail term in
\eqref{eq:s-binary-risk}.  For the derivative, note uniformly on the central
set that
\[
 x-\{t_0+\mu_J/\sqrt n\}^2=V_x+O(H/\sqrt n).
\]
Squaring \eqref{eq:s-bin-moments}, using
\eqref{eq:s-bin-tv2}--\eqref{eq:s-bin-tv3}, and bounding both posterior
variances by \(H^2\), gives the \(O(H^7/\sqrt n)\) central error and the
\(O(H^4e^{-c_0H^2})\) tail error in
\eqref{eq:s-binary-slope} and \eqref{eq:s-binary-centered}.

It remains to justify the leave-one-out comparison in
\eqref{eq:s-binary-slope}--\eqref{eq:s-binary-centered}.  Put
\(m=n-1\), \(\lambda_n=m/n\), and define
\[
 T^-_{n,x}(J_-)
 =\frac{2J_--m}{c\sqrt n}-\frac{m}{\sqrt n}t_0.
\]
Under the local parameter \(h\),
\[
 \E_hT^-_{n,x}=\lambda_n h,
 \qquad
 \Var_0(T^-_{n,x})=\lambda_n V_x.
\]
Let \(\mathcal G^-_{n,x}\) be the Gaussian experiment
\[
 Y^-_{n,x}=\lambda_n h+\sqrt{\lambda_n V_x}\,Z,
 \qquad Z\sim N(0,1).
\]
Its log-likelihood ratio relative to \(h=0\) is exactly
\[
 \log\frac{\dd G^-_{n,x,h}}{\dd G^-_{n,x,0}}(y)
 =\frac{hy}{V_x}-\frac{\lambda_nh^2}{2V_x}.
\]

The binomial likelihood based on \(J_-\) has the same two leading terms.
Indeed, substituting
\(J_-=mp_0+c\sqrt n\,y/2\) into its exact likelihood ratio and expanding to
third order gives
\[
 \log\frac{\Pp_h^{(m)}(J_-)}{\Pp_0^{(m)}(J_-)}
 =\frac{hy}{V_x}-\frac{\lambda_nh^2}{2V_x}
   +R^-_{n,x}(h,y),
\]
where, uniformly for \(x\in K\),
\[
 |R^-_{n,x}(h,y)|
 \leq C\left\{
   \frac{|h|^2|y|+|h|^3}{\sqrt n}+\frac{h^2}{n}
 \right\}.
\]
The lattice spacing remains \(a_n=2/(c\sqrt n)\), and Stirling's formula
now yields
\[
 \Pp_0^{(m)}\{T^-_{n,x}=y\}
 =a_n\phi_{\lambda_nV_x}(y)
  \left[1+O\{(1+|y|^3)/\sqrt n\}\right]
\]
uniformly for \(|y|=o(n^{1/6})\).  Consequently, on
\(\mathcal C_n^- =\{y:|y|\leq C_1H\}\), the binomial mass is within multiplicative
factors \(e^{\pm\eps_n^-}\) of
\(a_n\phi_{\lambda_nV_x}(y-\lambda_nh)\), where
\[
 \eps_n^-\leq C(H^3/\sqrt n+H^2/n).
\]
The binomial and Gaussian probabilities outside \(\mathcal C_n^-\) are bounded by
\(Ce^{-c_0H^2}\), uniformly in \(|h|\leq H\).

These inequalities remain valid after mixing over an arbitrary prior.
Dividing the mixed likelihoods by their predictive normalizing constants
shows that, on \(\mathcal C_n^-\), the binary and \(\mathcal G^-_{n,x}\) posterior density
ratios lie between \(e^{-2\eps_n^-}\) and
\(e^{2\eps_n^-}\).  The posterior-mean and posterior-variance comparisons in
\eqref{eq:s-bin-moments}, as well as the rectangle-rule bounds
\eqref{eq:s-bin-tv2}--\eqref{eq:s-bin-tv3}, therefore apply with \(J\)
replaced by \(J_-\) and with the same stated orders.

Finally, dividing \(Y^-_{n,x}\) by \(\lambda_n\) identifies
\(\mathcal G^-_{n,x}\) with a Gaussian location experiment of variance
\(V_x/\lambda_n\).  On \(|h|,|y|\leq C H\), its log-likelihood ratio differs
from that of variance \(V_x\) by at most \(CH^2/n\).  The corresponding
Gaussian benchmark densities also satisfy
\[
 \left|
 \log\frac{\phi_{V_x/\lambda_n}(y-h)}
               {\phi_{V_x}(y-h)}
 \right|
 \leq \frac{C(1+H^2)}{n}.
\]
Thus the full conditional densities, and not only their likelihood ratios,
are multiplicatively comparable.  Mixing over any prior preserves this
comparison, and posterior normalization at most doubles its exponent.  Since
the centered and uncentered squared normalized posterior-variance functionals
are bounded by \(CH^4\), their predictive expectations differ by
\(O(H^6/n)\); the tails remain exponentially small.  Together with
\[
 x-\{t_0+\mu_{J_-}/\sqrt n\}^2=V_x+O(H/\sqrt n)
\]
on the central set, this proves the leave-one-out bounds in
\eqref{eq:s-binary-slope} and \eqref{eq:s-binary-centered}.
\end{proof}

\begin{proposition}[Binary verification of the transfer conditions]
\label{prop:s-binary-GT}
The binary family satisfies the differentiated Gaussian-transfer assumption
of the main paper with the errors in
\eqref{eq:s-binary-risk}--\eqref{eq:s-binary-slope}.  Let
\(\delta_n\downarrow0\), and suppose that
\([x,x+\delta_n]\) remains in a common compact interior subset of \(K\).  If
\[
 H_n=\min\{\delta_n^{-1/20},n^{1/44}\},
\]
then
\[
 R^B_{n,H_n}(x+\delta_n)-R^B_{n,H_n}(x)
 =\delta_n\{1+o(1)\}
\]
uniformly on compact interior \(x\)-sets.
\end{proposition}

\begin{proof}
The compact minimax identity is \eqref{eq:s-minimax}.
 Lemma~\ref{lem:s-binary-derivative} gives the required nonnegative derivative.
The square-root integrand in that representation is bounded by \(CH^2\) on
compact inverse-information sets, so the derivative is bounded by \(CH^4\).
Lemma~\ref{lem:s-binary-transfer} gives the
all-prior risk and derivative comparisons.  The slow radius satisfies
\[
 H_n^{11}/\sqrt n\to0,\qquad H_n^{10}\delta_n\to0,
\]
so the differential transfer theorem in the main paper applies.
\end{proof}

\section{Poisson thinning}
\label{sec:s-poisson}

Fix \(\theta_0>0\).  At degradation level \(x>0\), observe
\[
 Y_i\sim\operatorname{Poisson}\{(\theta_0+h/\sqrt n)/x\},
 \qquad |h|\leq H.
\]
The total count \(J=\sum_iY_i\) has conditional mean
\[
 \Lambda_h/x,\qquad \Lambda_h=n\theta_0+\sqrt n\,h,
\]
and the inverse Fisher information is \(v=x\theta_0\).

\subsection{Exact derivative along the thinning semigroup}

\begin{lemma}[Poisson posterior derivative]
\label{lem:s-poisson-derivative}
Let \(\mu_j\) and \(w_j\) be the posterior mean and variance after total count
\(j\).  Then
\begin{equation}
 \partial_vr^P_{n,v}(\pi)
 =\E_{\pi,v}\left[
   \frac{w_J}{v\sqrt{1+\mu_J/(\theta_0\sqrt n)}}
 \right]^2.
 \label{eq:s-poisson-derivative}
\end{equation}
\end{lemma}

\begin{proof}
Put \(s=\log x\), and let \(p_j(s\mid h)\) be the conditional total-count
probability.  The Poisson pure-death equation is
\[
 \partial_sp_j(s\mid h)=(j+1)p_{j+1}(s\mid h)-jp_j(s\mid h).
\]
Let \(M_j=\int p_j(s\mid h)\pi(\dd h)\), let
\(N_j=\int h p_j(s\mid h)\pi(\dd h)\), and put \(m_j=N_j/M_j\).
Differentiation of the posterior explained variance
\(\sum_jN_j^2/M_j\), followed by pairing the mass moving from \(j+1\) to
\(j\), gives
\begin{equation}
 \partial_sr^P_{n,v}(\pi)
 =\sum_{j\geq1}jM_j(m_j-m_{j-1})^2.
 \label{eq:s-poisson-dirichlet}
\end{equation}
The posterior at count \(j\) is the posterior at count \(j-1\) tilted by
the affine function \(\Lambda_h\).  Therefore
\[
 m_j-m_{j-1}
 =\frac{\sqrt n\,w_{j-1}}
        {n\theta_0+\sqrt n\,\mu_{j-1}}.
\]
The Poisson recursion also gives
\[
 jM_j=\frac{n}{x}\{\theta_0+\mu_{j-1}/\sqrt n\}M_{j-1}.
\]
Since \(\partial_v=v^{-1}\partial_s\), substitution in
\eqref{eq:s-poisson-dirichlet} and an index shift prove
\eqref{eq:s-poisson-derivative}.
\end{proof}

\subsection{Uniform Poisson-to-Gaussian comparison}

Define
\begin{equation}
 T_{n,x}(J)=\frac{xJ-n\theta_0}{\sqrt n}.
 \label{eq:s-poisson-statistic}
\end{equation}
At \(h=0\), this statistic has variance \(v=x\theta_0\).

\begin{lemma}[All-prior Poisson transfer]
\label{lem:s-poisson-transfer}
Let \(\theta_0\) and \(x\) range over compact subsets of
\((0,\infty)\).  Whenever \(1\leq H=o(n^{1/6})\), uniformly over all priors
on \([-H,H]\),
\begin{align}
 |r^P_{n,v}(\pi)-r_v^G(\pi)|
 &\leq CH^5n^{-1/2}+CH^2e^{-cH^2},                     \label{eq:s-poisson-risk}\\
 \left|\partial_vr^P_{n,v}(\pi)
       -\E_\pi^G\{w_\pi(Y)/v\}^2\right|
 &\leq CH^7n^{-1/2}+CH^4e^{-cH^2}.                     \label{eq:s-poisson-slope}
\end{align}
\end{lemma}

\begin{proof}
It suffices to consider \(H\geq1\).  The statistic in
\eqref{eq:s-poisson-statistic} lies on
\[
 \mathcal L_n=\{xj/\sqrt n-\theta_0\sqrt n:j=0,1,\ldots\}
\]
with spacing \(a_n=x/\sqrt n\).  Let
\(\mathcal C_n=\{y\in\mathcal L_n:|y|\leq C_0H\}\), for a fixed \(C_0>1\).

At \(T_{n,x}(J)=y\), direct expansion of the likelihood ratio gives
\begin{align}
 \log\frac{p_h(J)}{p_0(J)}
 &=\frac{hy}{v}-\frac{h^2}{2v}+R_{n,x}(h,y),             \label{eq:s-pois-lr}\\
 |R_{n,x}(h,y)|
 &\leq C\frac{|h|^2|y|+|h|^3}{\sqrt n}.                 \notag
\end{align}
Indeed, substitute
\(J=n\theta_0/x+\sqrt n\,y/x\) into
\[
 J\log\{1+h/(\theta_0\sqrt n)\}-\sqrt n\,h/x.
\]
On \(\mathcal C_n\), Stirling's formula with remainder gives
\begin{equation}
 \left|\log\frac{p_0(J)}{a_n\phi_v(y)}\right|
 \leq C\frac{1+|y|^3}{\sqrt n}.
 \label{eq:s-pois-local-limit}
\end{equation}
For example, with
\(\lambda_0=n\theta_0/x\) and \(z=(J-\lambda_0)/\sqrt{\lambda_0}=y/\sqrt v\),
the first terms are
\[
 \log\frac{p_0(J)}{a_n\phi_v(y)}
 =\frac{z^3-3z}{6\sqrt{\lambda_0}}
 +O\{(1+z^2+z^4)/\lambda_0\}.
\]
Equations \eqref{eq:s-pois-lr}--\eqref{eq:s-pois-local-limit} imply
\begin{equation}
 e^{-\eps_n}a_n\phi_v(y-h)
 \leq p_h(J)\leq e^{\eps_n}a_n\phi_v(y-h),
 \qquad \eps_n=CH^3/\sqrt n=o(1),
 \label{eq:s-pois-relative}
\end{equation}
uniformly on \(\mathcal C_n\).

The two-sided Poisson Chernoff inequality gives
\[
 \sup_{|h|\leq H}\Pp_h\{T_{n,x}\notin\mathcal C_n\}
 \leq Ce^{-cH^2}.
\]
The same bound holds for the Gaussian experiment, and mixture over any prior
preserves it.  Integrating \eqref{eq:s-pois-relative} over the prior and
normalizing shows that the Poisson posterior, relative to the Gaussian
posterior at \(y\), has density between \(e^{-2\eps_n}\) and
\(e^{2\eps_n}\).  Hence, on \(\mathcal C_n\),
\begin{equation}
 |\mu_J-\mu_\pi(y)|\leq CH^4/\sqrt n,\qquad
 |w_J-w_\pi(y)|\leq CH^5/\sqrt n.
 \label{eq:s-pois-moments}
\end{equation}

For every absolutely continuous integrable \(f\), the half-lattice obeys
\begin{equation}
 \left|a_n\sum_{y\in\mathcal L_n}f(y)-\int_\R f(y)\,\dd y\right|
 \leq a_n\int_\R|f'(y)|\,\dd y
 +\int_{-\infty}^{-\theta_0\sqrt n}|f(y)|\,\dd y.
 \label{eq:s-half-lattice}
\end{equation}
For the Gaussian prior mixtures below, the second term is polynomial in
\(H\) times \(e^{-cn}\).  Gaussian posterior differentiation gives
\[
 m_\pi'=m_\pi(\mu_\pi-y)/v,\qquad
 \mu_\pi'=w_\pi/v,\qquad
 w_\pi'=\kappa_{3,\pi}/v.
\]
Using \(w_\pi\leq H^2\), \(|\kappa_{3,\pi}|\leq CH^3\), and integrating
these identities yields
\begin{align}
 \int_\R|\partial_y(m_\pi w_\pi)|\,\dd y
 &\leq C(1+H^3),                                       \label{eq:s-pois-tv1}\\
 \int_\R\left|\partial_y[
 m_\pi\{w_\pi/v\}^2]\right|\,\dd y
 &\leq C(1+H^5).                                       \label{eq:s-pois-tv2}
\end{align}

The Poisson Bayes risk is \(\sum_JM_Jw_J\).  On the central set,
\eqref{eq:s-pois-relative}, \eqref{eq:s-pois-moments}, and
\eqref{eq:s-half-lattice}--\eqref{eq:s-pois-tv1} compare it with
\(\int m_\pi w_\pi\) at error \(O(H^5/\sqrt n)\).  The tails contribute
\(O(H^2e^{-cH^2})\), proving \eqref{eq:s-poisson-risk}.

For \eqref{eq:s-poisson-slope}, square the posterior-variance comparison in
\eqref{eq:s-pois-moments}.  Since both variances are at most \(H^2\), the
central error is \(O(H^7/\sqrt n)\), including the predictive-mass error.
Furthermore,
\[
 \left|\{1+\mu_J/(\theta_0\sqrt n)\}^{-1}-1\right|
 \leq CH/\sqrt n.
\]
Thus replacing the denominator in
\eqref{eq:s-poisson-derivative} contributes \(O(H^5/\sqrt n)\).  The
derivative integrands are bounded by \(CH^4\), so the two tails contribute
\(CH^4e^{-cH^2}\).  Finally use
\eqref{eq:s-half-lattice} and \eqref{eq:s-pois-tv2}.
\end{proof}

\begin{proposition}[Poisson verification of the transfer conditions]
\label{prop:s-poisson-GT}
The Poisson family satisfies the differentiated Gaussian-transfer assumption
of the main paper with the errors in
\eqref{eq:s-poisson-risk}--\eqref{eq:s-poisson-slope}.  Let
\(\delta_n\downarrow0\), and suppose that
\([v,v+\delta_n]\) remains in a common compact interior subset of
\((0,\infty)\).  With
\[
 H_n=\min\{\delta_n^{-1/20},n^{1/44}\},
\]
its local minimax risk satisfies
\[
 R^P_{n,H_n}(v+\delta_n)-R^P_{n,H_n}(v)
 =\delta_n\{1+o(1)\}
\]
uniformly on compact interior \(v\)-sets.
\end{proposition}

\begin{proof}
 Equation \eqref{eq:s-minimax} gives Bayes--minimax duality.
Lemma~\ref{lem:s-poisson-derivative} gives the nonnegative posterior derivative;
on the stated local range, its square-root integrand is bounded by \(CH^2\),
so the derivative is bounded by \(CH^4\).  Lemma~\ref{lem:s-poisson-transfer}
gives the all-prior Gaussian comparison.  The exponent conditions are the same as in
Proposition~\ref{prop:s-binary-GT}.
\end{proof}

\section{Primitive likelihood-generator conditions}
\label{sec:s-primitive}

This section derives the all-prior comparisons from conditions on the
experiment laws and their degradation generator.  The conditions do not
involve a prior, a posterior distribution, or a least-favorable sequence.

Let \(\cK\Subset(0,\infty)\) be a compact interval of inverse-information
levels.  After a fixed measurable identification, suppose that the experiment
at level \(v\in\cK\) has a common countable state space \(\mathcal Z_n\) and
mass function
\[
 p_{n,v,h}(z),\qquad z\in\mathcal Z_n,\quad |h|\leq H.
\]
Increasing \(v\) is assumed to be a parameter-free Markov degradation with
generator
\[
 \mathcal L_{n,v}f(z)
 =\sum_{z'}q_{n,v}(z,z')\{f(z')-f(z)\},
\]
so that \(\partial_vp_{n,v,h}=p_{n,v,h}\mathcal L_{n,v}\).
We assume common positive support along every active generator edge.

Let \(t_{n,v}:\mathcal Z_n\to\R\) embed the sufficient statistic into the
Gaussian coordinate.  Associate \(z\) with a quadrature cell
\(Q_{n,v}(z)\), whose length is \(a_{n,v}(z)\), such that the cell interiors
are disjoint and \(t_{n,v}(z)\in\overline{Q_{n,v}(z)}\).  Their union, up to
Lebesgue-null endpoints, is an interval \(I_{n,v}\).  Write
\[
 g_{v,h}(y)=\phi_v(y-h),\qquad
 \mathcal C_{n,H}=\{z:|t_{n,v}(z)|\leq C_0H\},
\]
where \(\phi_v\) is the \(N(0,v)\) density.  All constants below are uniform
over \(v\in\cK\).

\medskip\noindent\textbf{PLG1: multiplicative local Gaussian comparison.}
There is an \(\eps_n(H)\leq1\) such that, for every \(|h|\leq H\) and
\(z\in\mathcal C_{n,H}\),
\begin{equation}
 \left|
 \log\frac{p_{n,v,h}(z)}
 {a_{n,v}(z)g_{v,h}\{t_{n,v}(z)\}}
 \right|\leq\eps_n(H).
 \label{eq:s-PLG1}
\end{equation}

\medskip\noindent\textbf{PLG2: probability and energy tails.}
For a deterministic \(\tau_n(H)\),
\begin{align}
 &\sup_{|h|\leq H}\Pp_{n,v,h}(\mathcal C_{n,H}^c)
 +\sup_{|h|\leq H}\Pp^G_{v,h}(|Y|>C_0H) \notag\\
 &\quad
 +\sup_{|h|\leq H}\sum_{z\notin\mathcal C_{n,H}}
 a_{n,v}(z)g_{v,h}\{t_{n,v}(z)\}
 \leq\tau_n(H).
 \label{eq:s-PLG2a}
\end{align}
The energy-tail requirement is stated after the edge envelope is defined
below.

\medskip\noindent\textbf{PLG3: prior-free quadrature.}
Suppose
\[
 \sup_z\operatorname{diam}\{Q_{n,v}(z)\}\leq\kappa_n.
\]
Then every absolutely continuous integrable \(f\) satisfies
\begin{equation}
 \left|
 \sum_z a_{n,v}(z)f\{t_{n,v}(z)\}
 -\int_{I_{n,v}}f(y)\,\dd y
 \right|
 \leq\kappa_n\int_{I_{n,v}}|f'(y)|\,\dd y.
 \label{eq:s-PLG3a}
\end{equation}
We also require
\begin{equation}
 \sup_{|h|\leq H}\Pp^G_{v,h}(Y\notin I_{n,v})\leq b_n(H).
 \label{eq:s-PLG3b}
\end{equation}
The inequality in \eqref{eq:s-PLG3a} follows by applying the fundamental
theorem of calculus on each cell.  For a uniform lattice
\(y_j=y_0+ja_n\), we use the left cells
\(Q_j=[y_j,y_j+a_n)\).  A finite lattice \(j=0,\ldots,N\) therefore covers
\([y_0,y_N+a_n]\), and a half-lattice \(j\geq0\) covers
\([y_0,\infty)\).  This convention counts every endpoint mass once.

\medskip\noindent\textbf{PLG4: edge score and calibration.}
For an active edge \(e=(z,z')\), define
\[
 \ell_e(h)=\log\frac{p_{n,v,h}(z')}{p_{n,v,h}(z)}
 =\ell_e(0)+\alpha_eh+\rho_e(h),
 \qquad \rho_e(0)=\rho_e'(0)=0,
\]
and let
\[
 \omega_e(H)=\sup_{|h|\leq H}|\rho_e(h)|,\qquad
 B_e(H)=\alpha_e^2H^2+\omega_e(H).
\]
Assume that
\begin{equation}
 |\alpha_e|H+2\omega_e(H)\leq c_0
 \label{eq:s-PLG4a}
\end{equation}
for every active edge with origin in \(\mathcal C_{n,H}\), where \(c_0\)
is fixed.  Put
\begin{equation}
 \Xi_{n,v,H}(z)
 =\sum_{z'}q_{n,v}(z,z')
 \left\{|\alpha_e|H^3B_e(H)+H^2B_e(H)^2\right\}.
 \label{eq:s-Xi}
\end{equation}
On the central set, require
\begin{align}
 \left|\sum_{z'}q_{n,v}(z,z')\alpha_e^2-v^{-2}\right|
 &\leq\gamma_n(H),                                      \label{eq:s-PLG4b}\\
 \Xi_{n,v,H}(z)&\leq\chi_n(H).                           \label{eq:s-PLG4c}
\end{align}

To complete PLG2, let
\[
 \Delta_e(H)=\osc_{|h|\leq H}\ell_e(h)
\]
and define the deterministic edge-oscillation envelope
\begin{equation}
 G_{n,v,H}(z)
 =4H^2\sum_{z'}q_{n,v}(z,z')
       \min\{1,[\exp\{\Delta_e(H)\}-1]^2\}.
 \label{eq:s-energy-envelope}
\end{equation}
We assume
\begin{equation}
 \sup_{|h|\leq H}
 \E_{n,v,h}\{\ind_{\mathcal C_{n,H}^c}
             G_{n,v,H}(Z)\}
 \leq\tau_n^\Gamma(H).
 \label{eq:s-PLG2b}
\end{equation}
Every quantity in PLG1--PLG4 is determined by the sampling laws and the
degradation generator.

\begin{remark}[A regular-rate sufficient condition]
\label{rem:s-regular-PLG}
Suppose that, on the central set,
\[
 \sum_{z'}q_{n,v}(z,z')\leq Cn,\qquad
 \max_{z':q(z,z')>0}|\alpha_{z,z'}|\leq Cn^{-1/2},
\qquad
 \max_{z':q(z,z')>0}\omega_{z,z'}(H)\leq CH^2/n.
\]
If \(H=o(\sqrt n)\), then \eqref{eq:s-PLG4a} holds for all sufficiently
large \(n\), and
\[
 \Xi_{n,v,H}(z)
 \leq C\{H^5/\sqrt n+H^6/n\}.
\]
Thus the remaining central calculation is the scalar calibration in
\eqref{eq:s-PLG4b}.  Binary, Poisson, and negative-binomial thinning all
have this regular local scale.
\end{remark}

\begin{lemma}[Uniform edge-tilt expansion]
\label{lem:s-edge-tilt}
Let \(P\) be a probability law on \([-H,H]\), with mean \(\mu\) and
variance \(w\), and let
\[
 \frac{\dd P_e}{\dd P}(h)
 \ \propto\ \exp\{\alpha h+\rho(h)\}.
\]
If \(|\alpha|H+\osc(\rho)\leq c\), then
\begin{equation}
 \left|\E_eh-\mu-\alpha w\right|
 \leq C_cH\{\alpha^2H^2+\osc(\rho)\}.
 \label{eq:s-edge-tilt}
\end{equation}
\end{lemma}

\begin{proof}
Subtracting a constant from \(\rho\) does not change the tilted law.  Center
at \(\mu\) and set
\[
 u(h)=\alpha(h-\mu)+\rho(h)-\rho(\mu).
\]
Then \(\|u\|_\infty\leq2c\), and Taylor's theorem, including normalization
by \(\E\exp(u)\), gives
\[
 \frac{\exp\{u(h)\}}{\E\exp(u)}
 =1+\alpha(h-\mu)+R(h),\qquad
 \|R\|_\infty
 \leq C_c\{\alpha^2H^2+\osc(\rho)\}.
\]
Taking covariance with \(h\), the linear term gives \(\alpha w\).  The
remaining covariance is bounded by the right-hand side of
\eqref{eq:s-edge-tilt}.  The normalization is bounded away from zero by a
constant depending only on \(c\).
\end{proof}

For a prior \(\pi\), write \(\mu_\pi^E(z)\) and \(w_\pi^E(z)\) for the
posterior mean and variance of \(h\).  Adjacent posteriors are related by the
edge tilt in Lemma~\ref{lem:s-edge-tilt}.  Hence
\begin{equation}
 \left|\mu_\pi^E(z')-\mu_\pi^E(z)-\alpha_ew_\pi^E(z)\right|
 \leq CHB_e(H).
 \label{eq:s-edge-posterior}
\end{equation}
Writing
\[
 \Gamma_{n,v}m(z)
 =\sum_{z'}q_{n,v}(z,z')\{m(z')-m(z)\}^2,
\]
expansion of the square in \eqref{eq:s-edge-posterior} yields
\begin{equation}
 \left|
 \Gamma_{n,v}\mu_\pi^E(z)
 -w_\pi^E(z)^2
  \sum_{z'}q_{n,v}(z,z')\alpha_e^2
 \right|
 \leq C\Xi_{n,v,H}(z).
 \label{eq:s-edge-energy}
\end{equation}
Here \(\Xi_{n,v,H}\) is the prior-free remainder in
\eqref{eq:s-Xi}.
The tail envelope in \eqref{eq:s-energy-envelope} follows from the same
tilt relation.  Indeed, two laws whose likelihood ratio has log-oscillation
\(\Delta\) differ in total variation by at most
\(\min\{1,\exp(\Delta)-1\}\).  Since \(|h|\leq H\), this bounds each squared
adjacent posterior-mean difference by the corresponding summand in
\eqref{eq:s-energy-envelope}.

\begin{lemma}[Gaussian-mixture variation]
\label{lem:s-gaussian-variation}
For a prior \(\pi\) on \([-H,H]\), let \(m_\pi^G\), \(\mu_\pi^G\),
\(w_\pi^G\), and \(\kappa_{3,\pi}^G\) denote the predictive density,
posterior mean, posterior variance, and posterior third central moment in
\(Y=h+\sqrt vZ\).  Uniformly over \(v\in\cK\),
\begin{align}
 \int\left|\partial_y(m_\pi^Gw_\pi^G)\right|\,\dd y
 &\leq C_{\cK}(1+H^3),                                  \label{eq:s-gauss-var1}\\
 \int\left|\partial_y\left[
 m_\pi^G\{w_\pi^G/v\}^2\right]\right|\,\dd y
 &\leq C_{\cK}(1+H^5).                                  \label{eq:s-gauss-var2}
\end{align}
Moreover,
\[
 \|m_\pi^Gw_\pi^G\|_\infty\leq C_{\cK}H^2,\qquad
 \|m_\pi^G(w_\pi^G/v)^2\|_\infty\leq C_{\cK}H^4.
\]
\end{lemma}

\begin{proof}
Gaussian posterior differentiation gives
\[
 (m_\pi^G)'=m_\pi^G(\mu_\pi^G-y)/v,\qquad
 (\mu_\pi^G)'=w_\pi^G/v,\qquad
 (w_\pi^G)'=\kappa_{3,\pi}^G/v.
\]
The support bound gives \(w_\pi^G\leq H^2\) and
\(|\kappa_{3,\pi}^G|\leq8H^3\).  Also,
\[
 \int m_\pi^G(y)|\mu_\pi^G(y)-y|\,\dd y
 \leq\E|h-Y|=\sqrt{2v/\pi}.
\]
The product rule proves \eqref{eq:s-gauss-var1}--\eqref{eq:s-gauss-var2}.
The sup-norm bounds follow from
\(m_\pi^G\leq(2\pi v)^{-1/2}\).
\end{proof}

\begin{theorem}[Primitive differentiated transfer]
\label{thm:s-primitive-transfer}
Assume PLG1--PLG4, the posterior Dirichlet identity in Section
\ref{sec:s-dirichlet}, and compact Bayes--minimax duality.  Uniformly over
all priors \(\pi\) on \([-H,H]\),
\begin{align}
 |r^E_{n,v}(\pi)-r_v^G(\pi)|
 &\leq\eta_{0,n}(H),                                    \label{eq:s-primitive-risk}\\
 \left|\partial_vr^E_{n,v}(\pi)
 -\E_\pi^G\{w_\pi^G(Y)/v\}^2\right|
 &\leq\eta_{1,n}(H),                                    \label{eq:s-primitive-derivative}
\end{align}
where
\begin{align}
 \eta_{0,n}(H)
 &\leq C\left[H^2\eps_n(H)+\kappa_n(1+H^3)
 +H^2\{\tau_n(H)+b_n(H)\}\right],                       \label{eq:s-primitive-eta0}\\
 \eta_{1,n}(H)
 &\leq C\left[H^4\eps_n(H)+\kappa_n(1+H^5)
 +H^4\{\tau_n(H)+b_n(H)+\gamma_n(H)\}\right. \notag\\
 &\hspace{33mm}\left.
 +\chi_n(H)+\tau_n^\Gamma(H)\right].                    \label{eq:s-primitive-eta1}
\end{align}
In addition,
\begin{equation}
 0\leq\partial_vr^E_{n,v}(\pi)
 \leq C_{\cK}H^4+C\eta_{1,n}(H).
 \label{eq:s-primitive-GT1}
\end{equation}
Thus, when \(H\geq1\) and \(\eta_{1,n}(H)\to0\), the primitive conditions
imply the differentiated Gaussian-transfer assumption in the main paper.
\end{theorem}

\begin{proof}
Fix a prior \(\pi\).  Mixing \eqref{eq:s-PLG1} over \(\pi\) preserves its
multiplicative bounds.  Dividing the conditional and predictive likelihood
comparisons shows that, on the central set, the experiment posterior relative
to the Gaussian posterior at \(t_{n,v}(z)\) has density between
\(\exp\{-2\eps_n(H)\}\) and \(\exp\{2\eps_n(H)\}\).  Since \(|h|\leq H\),
\begin{align}
 |\mu_\pi^E(z)-\mu_\pi^G\{t_{n,v}(z)\}|
 &\leq CH\eps_n(H),                                     \label{eq:s-post-mean}\\
 |w_\pi^E(z)-w_\pi^G\{t_{n,v}(z)\}|
 &\leq CH^2\eps_n(H).                                   \label{eq:s-post-var}
\end{align}

Bayes risk is predictive expected posterior variance.  The posterior
comparisons in \eqref{eq:s-post-mean}--\eqref{eq:s-post-var}, together with
\eqref{eq:s-PLG1} and
\eqref{eq:s-PLG2a} compare its central and tail parts with the discrete
Gaussian reference.  The discrete-reference tail in \eqref{eq:s-PLG2a}
allows the central sum to be extended to the full embedded lattice.
Equation \eqref{eq:s-PLG3a}, applied to \(m_\pi^Gw_\pi^G\), and
\eqref{eq:s-PLG3b} then compare the sum with the Gaussian integral.
Lemma~\ref{lem:s-gaussian-variation} gives
\eqref{eq:s-primitive-risk} and \eqref{eq:s-primitive-eta0}.

The posterior Dirichlet identity gives
\[
 \partial_vr^E_{n,v}(\pi)
 =\E_\pi^E\Gamma_{n,v}\mu_\pi^E(Z).
\]
On the central set, \eqref{eq:s-edge-energy} and
\eqref{eq:s-PLG4b}--\eqref{eq:s-PLG4c} imply
\begin{equation}
 \left|\Gamma_{n,v}\mu_\pi^E(z)
 -v^{-2}w_\pi^E(z)^2\right|
 \leq C\{H^4\gamma_n(H)+\chi_n(H)\}.
 \label{eq:s-central-energy}
\end{equation}
 Equation \eqref{eq:s-central-energy} controls the central contribution, and
the tail contribution is bounded by \eqref{eq:s-PLG2b}.
Equation \eqref{eq:s-post-var} compares the central posterior variances,
and the predictive likelihood comparison together with
\eqref{eq:s-PLG3a}--\eqref{eq:s-PLG3b} compares the resulting discrete
Gaussian sum with
\(\E_\pi^G\{w_\pi^G(Y)/v\}^2\).
Lemma~\ref{lem:s-gaussian-variation} now gives
\eqref{eq:s-primitive-derivative} and
\eqref{eq:s-primitive-eta1}.  Nonnegativity in
\eqref{eq:s-primitive-GT1} follows from the Dirichlet identity; its upper
bound follows from \eqref{eq:s-primitive-derivative} and
\((w_\pi^G/v)^2\leq C_{\cK}H^4\).  No bound on the pointwise jump rate is
used.
\end{proof}

\begin{lemma}[Monotone-tilt secant slope]
\label{lem:s-monotone-tilt}
Let \(P_1\) and \(P_2\) be posterior laws obtained from a common base law by
likelihood weights whose ratio is monotone in \(h\).  For a differentiable
target \(g\), whenever \(\E_2h-\E_1h\neq0\),
\[
 \frac{\E_2g-\E_1g}{\E_2h-\E_1h}
\]
lies between the smallest and largest secant slopes of \(g\) on
\([-H,H]\).  If the denominator is zero, the numerator is also zero.
\end{lemma}

\begin{proof}
Let \(H_1,H_2\) be independent draws from the common base law, and let
\(r\) be the ratio of the two likelihood weights.  Up to the same positive
normalizing factor, the numerator and denominator equal expectations of
\[
 \{g(H_1)-g(H_2)\}\{r(H_1)-r(H_2)\}
\quad\text{and}\quad
 (H_1-H_2)\{r(H_1)-r(H_2)\}.
\]
After multiplication by the monotonicity sign, the denominator weights are
nonnegative.  The ratio is therefore a convex average of secant slopes.  If
the denominator is zero, the nonnegative integrand vanishes almost surely,
which also makes the numerator zero.
\end{proof}

\begin{corollary}[Primitive smooth-target transfer]
\label{cor:s-primitive-ST}
Suppose the active edge likelihood ratios in PLG4 are monotone on
\([-H,H]\), and
\[
 g_n(h)=bh+\rho_n(h),\qquad
 \sup_{|h|\leq H}|\rho_n(h)|\leq CH^2/\sqrt n,\qquad
 \sup_{|h|\leq H}|\rho_n'(h)|\leq CH/\sqrt n.
\]
Then the derivative energies for \(g_n\) and \(h\) differ by the factors
\[
 \{|b|-CH/\sqrt n\}_+^2
 \quad\text{and}\quad
 \{|b|+CH/\sqrt n\}^2.
\]
The corresponding ordinary Bayes-risk comparison follows from conditional
variance and Cauchy--Schwarz.  Hence the smooth-target transfer assumption
follows from monotone likelihood edges and target smoothness.
\end{corollary}

\begin{proof}
Apply Lemma~\ref{lem:s-monotone-tilt} to every edge in the posterior
Dirichlet form.  Every secant slope of \(g_n\) lies between
\(b-CH/\sqrt n\) and \(b+CH/\sqrt n\).  Squaring the adjacent posterior-mean
increments and summing with the nonnegative jump rates gives the derivative
comparison.  For ordinary Bayes risk, write
\(g_n(h)=bh+\rho_n(h)\) inside the conditional variance and use
\(\Var\{\rho_n(h)\mid Y\}\leq C H^4/n\).
\end{proof}

\subsection{Primitive verification for the original models}
\label{sec:s-primitive-original}

For Gaussian degradation, the experiment is the reference experiment.  The
heat carré du champ is \((\partial_y\mu_\pi^G)^2\), and Gaussian posterior
differentiation gives \(\partial_y\mu_\pi^G=w_\pi^G/v\).  Thus the risk and
derivative comparisons are exact, with zero transfer error.

For binary attenuation, use the sufficient plus count and its birth--death
degradation.  Equations \eqref{eq:s-bin-lr}--\eqref{eq:s-bin-tail} give
PLG1--PLG3 with
\[
 \eps_n(H)\leq CH^3/\sqrt n,\qquad
 \kappa_n\leq C/\sqrt n,\qquad
 \tau_n(H)\leq Ce^{-cH^2},
\]
with an exponentially small uncovered Gaussian tail.  Put
\[
 p_h=\{1+c(t_0+h/\sqrt n)\}/2,\qquad
 x=c^{-2},\qquad v=x-t_0^2.
\]
In binary-channel time \(s\), the count generator has downward and upward
rates \(j\) and \(n-j\).  Since \(\dd v/\dd s=4x\), the total \(v\)-rate is
\(n/(4x)\), and the two edge slopes have common magnitude
\[
 |\alpha|=\frac{2c}{\sqrt n\{1-c^2t_0^2\}}.
\]
The edge remainders are bounded by \(CH^2/n\), and
\[
 \frac{n}{4x}
 \frac{4c^2}{n\{1-c^2t_0^2\}^2}
 =v^{-2}.
\]
Thus the calibration error is zero.  Remark~\ref{rem:s-regular-PLG} gives
\(\chi_n(H)\leq C\{H^5/\sqrt n+H^6/n\}\).  Moreover,
\(G_{n,v,H}\leq CH^4\), so the binomial tail bound gives
\(\tau_n^\Gamma(H)\leq CH^4e^{-cH^2}\).

For Poisson thinning, the likelihood expansion
\eqref{eq:s-pois-lr}, the local limit \eqref{eq:s-pois-local-limit}, and the
half-lattice comparison \eqref{eq:s-half-lattice} give PLG1--PLG3 with the
same orders.  The sufficient total count has pure-death generator
\[
 q_v(j,j-1)=\frac{j}{x\theta_0},
\]
and the adjacent likelihood ratio has
\[
 \alpha=-\frac{1}{\sqrt n\,\theta_0},\qquad
 \omega_e(H)\leq CH^2/n.
\]
At \(j_0=n\theta_0/x\),
\[
 q_v(j_0,j_0-1)\alpha^2=v^{-2},\qquad v=x\theta_0,
\]
while \(j=j_0+\sqrt n\,y/x\) gives a calibration error
\(O(H/\sqrt n)\) on the central set.  Finally,
\(G_{n,v,H}(j)\leq CH^4j/n\).  The bounded second moment of \(J/n\) and
the Poisson Chernoff bound give
\(\tau_n^\Gamma(H)\leq CH^4e^{-cH^2}\).  Substitution in
Theorem~\ref{thm:s-primitive-transfer} recovers
\eqref{eq:s-binary-risk}--\eqref{eq:s-binary-slope} and
\eqref{eq:s-poisson-risk}--\eqref{eq:s-poisson-slope}.

\subsection{Verification for negative-binomial thinning}
\label{sec:s-negative-binomial}

Fix an overdispersion parameter \(r>0\).  Let \(X_1,\ldots,X_n\) be
independent negative-binomial observations with shape \(r\) and mean
\(\theta_h/x\), where
\[
 \theta_h=\theta_0+h/\sqrt n,\qquad x\geq1.
\]
This family has probability generating function
\[
 \left\{\frac{xr}{xr+\theta_h(1-z)}\right\}^r.
\]
Binomial thinning from \(x_1\) to \(x_2>x_1\) uses the
parameter-free retention probability \(x_1/x_2\); this standard closure
property underlies negative-binomial thinning \citep{ZhuJoe2010}.
The sufficient total count \(J=\sum_iX_i\) has mass
\begin{equation}
 p_{n,x,h}(j)
 =\frac{\Gamma(nr+j)}{\Gamma(nr)j!}
  \left(\frac{xr}{xr+\theta_h}\right)^{nr}
  \left(\frac{\theta_h}{xr+\theta_h}\right)^j.
 \label{eq:s-NB-mass}
\end{equation}
The normalized statistic
\begin{equation}
 T_{n,x}(J)=\frac{xJ-n\theta_0}{\sqrt n}
 \label{eq:s-NB-statistic}
\end{equation}
has mean \(h\), and at \(h=0\) its variance is
\begin{equation}
 v(x)=x\theta_0+\theta_0^2/r,
 \label{eq:s-NB-v}
\end{equation}
the inverse Fisher information for \(\theta\).

\begin{proposition}[Negative-binomial primitive transfer]
\label{prop:s-negative-binomial-transfer}
Let \(\theta_0\), \(r\), and \(x\) range over compact subsets of
\((0,\infty)\), with \(x\geq1\).  If \(1\leq H=o(n^{1/6})\), then the
negative-binomial thinning experiment satisfies PLG1--PLG4 with
\begin{align}
 \eps_n(H)&\leq CH^3/\sqrt n,&
 \kappa_n&\leq C/\sqrt n,&
 \tau_n(H)&\leq Ce^{-cH^2},\notag\\
 b_n(H)&\leq Ce^{-cn},&
 \gamma_n(H)&\leq CH/\sqrt n,&
 \chi_n(H)&\leq C\{H^5/\sqrt n+H^6/n\},\notag\\
 \tau_n^\Gamma(H)&\leq CH^4e^{-cH^2}.&&&
 \label{eq:s-NB-PLG-rates}
\end{align}
Consequently, uniformly over all priors on \([-H,H]\),
\begin{align}
 |r^{NB}_{n,v}(\pi)-r_v^G(\pi)|
 &\leq CH^5/\sqrt n+CH^2e^{-cH^2},                       \label{eq:s-NB-risk}\\
 \left|\partial_vr^{NB}_{n,v}(\pi)
 -\E_\pi^G\{w_\pi^G(Y)/v\}^2\right|
 &\leq CH^7/\sqrt n+CH^4e^{-cH^2}.                       \label{eq:s-NB-derivative}
\end{align}
\end{proposition}

\begin{proof}
At the statistic in \eqref{eq:s-NB-statistic}, with
\(T_{n,x}(J)=y\) and \(v\) as in \eqref{eq:s-NB-v}, differentiation of
\eqref{eq:s-NB-mass} gives
\[
 \left.\partial_h\log p_{n,x,h}(J)\right|_{h=0}=y/v.
\]
Two further derivatives, uniformly on the stated compact parameter sets,
give
\begin{equation}
 \log\frac{p_{n,x,h}(J)}{p_{n,x,0}(J)}
 =\frac{hy}{v}-\frac{h^2}{2v}+R_n(h,y),\qquad
 |R_n(h,y)|
 \leq C\frac{|h|^2|y|+|h|^3}{\sqrt n}.
 \label{eq:s-NB-likelihood}
\end{equation}

For the local limit, write \(j=nt\), \(t_0=\theta_0/x\), and
\[
 F(t)=(r+t)\log(r+t)-r\log r-t\log t
 +r\log\frac{xr}{xr+\theta_0}
 +t\log\frac{\theta_0}{xr+\theta_0}.
\]
Stirling expansion of the three gamma factors in
\eqref{eq:s-NB-mass} gives \(nF(t)\), a smooth logarithmic prefactor, and a
uniform \(O(n^{-1})\) remainder near \(t_0\).  Direct differentiation gives
\[
 F'(t_0)=0,\qquad F''(t_0)=-x^2/v,\qquad
 \sup_{|t-t_0|\leq\delta}|F'''(t)|\leq C.
\]
With \(t=t_0+y/(x\sqrt n)\), the lattice spacing is
\(a_n=x/\sqrt n\), and
\begin{equation}
 p_{n,x,0}(J)
 =a_n\phi_v(y)
 \left[1+O\{(1+|y|^3)/\sqrt n\}\right]
 \label{eq:s-NB-local-limit}
\end{equation}
uniformly for \(|y|=o(n^{1/6})\).  The prefactor at \(t_0\) is
\(a_n/\sqrt{2\pi v}\); its logarithm changes by
\(O\{(1+|y|)/\sqrt n\}\), while the cubic Taylor term is
\(O(|y|^3/\sqrt n)\).  Combining
\eqref{eq:s-NB-likelihood}--\eqref{eq:s-NB-local-limit} with the Gaussian
likelihood ratio proves PLG1.

The negative-binomial moment generating function is finite in a common
neighborhood of zero.  A uniform Chernoff bound therefore gives
\begin{equation}
 \sup_{|h|\leq H}
 \Pp_{n,x,h}\{|T_{n,x}|>C_0H\}
 \leq Ce^{-cH^2}
 \label{eq:s-NB-tail}
\end{equation}
for fixed \(C_0>1\) and \(H=o(\sqrt n)\).  The half-lattice has lower
endpoint \(-\theta_0\sqrt n\) and spacing \(a_n=x/\sqrt n\).
The left cells cover
\([-\theta_0\sqrt n,\infty)\), so PLG3 holds with
\(\kappa_n\leq C/\sqrt n\) and \(b_n(H)\leq Ce^{-cn}\).
The same rectangle bound and Gaussian tail estimate establish the discrete
reference-tail term in PLG2.

Put \(s=\log x\).  Increasing \(s\) is independent thinning, and the total
count has pure-death generator \(q_s(j,j-1)=j\).  Since
\[
 \frac{\dd v}{\dd s}=x\theta_0,\qquad
 q_v(j,j-1)=\frac{j}{x\theta_0},
\]
the edge log-likelihood ratio has \(h\)-dependent part
\[
 \log(xr+\theta_h)-\log\theta_h.
\]
Consequently,
\begin{equation}
 \alpha
 =-\frac{xr}{\sqrt n\,\theta_0(xr+\theta_0)},\qquad
 \omega_e(H)\leq CH^2/n.
 \label{eq:s-NB-edge}
\end{equation}
At the central count \(j_0=n\theta_0/x\),
\begin{equation}
 q_v(j_0,j_0-1)\alpha^2
 =\frac{r^2}{\theta_0^2(xr+\theta_0)^2}
 =v^{-2}.
 \label{eq:s-NB-calibration}
\end{equation}
For \(|y|\leq C_0H\), substituting
\(j=j_0+\sqrt n\,y/x\) changes the left side of
\eqref{eq:s-NB-calibration} by \(O(H/\sqrt n)\).
Equations \eqref{eq:s-NB-edge}--\eqref{eq:s-NB-calibration} and
Remark~\ref{rem:s-regular-PLG} give the stated \(\gamma_n\) and \(\chi_n\)
rates.

It remains to control the energy tail because the pure-death rate is
unbounded.  The edge log-likelihood oscillation is \(O(H/\sqrt n)\), and
\[
 G_{n,v,H}(j)\leq CH^4j/n.
\]
The second moment of \(J/n\) is uniformly bounded.  Cauchy--Schwarz and
\eqref{eq:s-NB-tail}, with a change in \(c\), give
\[
 \sup_{|h|\leq H}\E_{n,x,h}
 \{\ind_{\mathcal C_{n,H}^c}G_{n,v,H}(J)\}
 \leq CH^4e^{-cH^2}.
\]
This completes PLG2--PLG4.  Substitution of
\eqref{eq:s-NB-PLG-rates} into
\eqref{eq:s-primitive-eta0}--\eqref{eq:s-primitive-eta1} proves
\eqref{eq:s-NB-risk}--\eqref{eq:s-NB-derivative}.  The edge likelihood
ratio in \eqref{eq:s-NB-edge} is monotone in \(h\), so
Corollary~\ref{cor:s-primitive-ST} also gives the smooth-target transfer.
\end{proof}

\section{Smooth-target slope transfer}
\label{sec:s-smooth}

Let
\[
 g_n(h)=\sqrt n\{\psi(\theta_0+h/\sqrt n)-\psi(\theta_0)\},
 \qquad b=\dot\psi(\theta_0),
\]
where \(\psi\) has a bounded second derivative on a fixed neighborhood of
\(\theta_0\).  Every secant slope of \(g_n\) on \([-H,H]\) then lies between
\begin{equation}
 b-CH/\sqrt n\quad\text{and}\quad b+CH/\sqrt n.
 \label{eq:s-secant-range}
\end{equation}

\begin{lemma}[Posterior secant slopes]
\label{lem:s-secant}
Let \(H_1,H_2\) be conditionally independent draws from a posterior
distribution on \([-H,H]\).  If \(\Var(H_1\mid Y)>0\), then
\begin{equation}
 \frac{\Cov\{g_n(H_1),H_1\mid Y\}}
      {\Var(H_1\mid Y)}
 =
 \frac{\E[\{g_n(H_1)-g_n(H_2)\}(H_1-H_2)\mid Y]}
      {\E[(H_1-H_2)^2\mid Y]}.
 \label{eq:s-secant-identity}
\end{equation}
The right-hand side is a weighted average of secant slopes and therefore lies
in the interval in \eqref{eq:s-secant-range}.  If the posterior variance is
zero, all conditional differences in \eqref{eq:s-secant-identity} vanish.
\end{lemma}

\begin{proof}
Conditional symmetrization gives
\[
 2\Cov\{g_n(H_1),H_1\mid Y\}
 =\E[\{g_n(H_1)-g_n(H_2)\}(H_1-H_2)\mid Y],
\]
and the same identity with \(g_n(h)=h\) gives twice the conditional variance.
After writing the numerator as a secant slope times \((H_1-H_2)^2\), the
weights are nonnegative and sum to one.
\end{proof}

\begin{proposition}[Derivative comparison for smooth targets]
\label{prop:s-smooth-transfer}
In each of the Gaussian, binary, Poisson, and negative-binomial experiments,
for every prior,
\begin{equation}
 \{\,|b|-CH/\sqrt n\,\}_+^2\,\partial_vr^h_{n,v}(\pi)
 \leq \partial_vr^{g_n}_{n,v}(\pi)
 \leq \{\,|b|+CH/\sqrt n\,\}^2\,\partial_vr^h_{n,v}(\pi).
 \label{eq:s-smooth-derivative}
\end{equation}
Moreover,
\begin{equation}
 |r^{g_n}_{n,v}(\pi)-b^2r^h_{n,v}(\pi)|
 \leq C\{H^3/\sqrt n+H^4/n\}.
 \label{eq:s-smooth-risk}
\end{equation}
Together with \eqref{eq:s-minimax}, this proves the smooth-target transfer
assumption in the main paper for all four models.
\end{proposition}

\begin{proof}
For the Gaussian experiment, direct differentiation of the posterior
integrals gives, for every square-integrable target \(g\),
\begin{equation}
 \partial_vr^g_v(\pi)
 =v^{-2}\E_\pi\Cov\{g(h),h\mid Y\}^2.
 \label{eq:s-gaussian-general-derivative}
\end{equation}
Apply Lemma~\ref{lem:s-secant} to compare the conditional covariance in
\eqref{eq:s-gaussian-general-derivative} with the posterior variance.

For the binary experiment, the hypercube Dirichlet form
\eqref{eq:s-hypercube-dirichlet} remains valid when the posterior mean of
\(h\) is replaced by that of \(g_n(h)\).  Conditional on \(n-1\)
observations, appending one observation tilts the posterior by the affine
function
\[
 p(h)=\{1+c(t_0+h/\sqrt n)\}/2.
\]
When the conditional variance is nonzero, the ratio of the adjacent
posterior-mean differences is
\[
 \frac{\Cov\{g_n(h),p(h)\mid J_-\}}
      {\Cov\{h,p(h)\mid J_-\}}
 =
 \frac{\Cov\{g_n(h),h\mid J_-\}}
      {\Var(h\mid J_-)}.
\]
Lemma~\ref{lem:s-secant} bounds this ratio.  If the variance is zero, both
adjacent differences vanish.  Substitution into the same nonnegative
Dirichlet sum gives \eqref{eq:s-smooth-derivative}.

For Poisson thinning, the pure-death calculation leading to
\eqref{eq:s-poisson-dirichlet} gives, for a general target,
\[
 \partial_vr^g_{n,v}(\pi)
 =v^{-1}\sum_{j\geq1}jM_j(\nu_j-\nu_{j-1})^2,
 \qquad \nu_j=\E\{g(h)\mid J=j\}.
\]
The posterior at \(j\) is the posterior at \(j-1\) tilted by the affine
function \(\Lambda_h=n\theta_0+\sqrt n\,h\).  The ratio of the target and
identity adjacent differences is again the conditional covariance divided
by the conditional variance.  Lemma~\ref{lem:s-secant} completes the
derivative comparison.

For negative-binomial thinning, adjacent total-count posteriors are related
by the monotone edge likelihood ratio in \eqref{eq:s-NB-edge}.  Applying
Lemma~\ref{lem:s-monotone-tilt} in the pure-death Dirichlet form gives the
same derivative comparison, as recorded in
Corollary~\ref{cor:s-primitive-ST}.

Finally, Taylor's theorem gives
\[
 g_n(h)=bh+\rho_n(h),\qquad
 \sup_{|h|\leq H}|\rho_n(h)|\leq CH^2/\sqrt n.
\]
Expanding the conditional variance and applying Cauchy--Schwarz yields
\eqref{eq:s-smooth-risk}.
\end{proof}

\begin{remark}
The least-favorable priors for \(h\) and for \(g_n(h)\) need not coincide.
Equation \eqref{eq:s-smooth-risk} makes a least-favorable prior for the smooth
target \(O(H^3/\sqrt n+H^4/n)\)-least favorable for the identity target,
after division by \(b^2\).  Under the slow radius used in the main paper,
multiplication of this certificate error by \(H^6\) tends to zero.  Hence the
Gaussian posterior-variance certificate applies to the actual endpoint
priors of the smooth-target minimax problem.
\end{remark}

\section{Coupled radial-qubit baseline}
\label{sec:s-qubit-baseline}

We prove the coupled baseline theorem in the main paper.  For an equatorial
size-\(k\) menu let
\begin{align*}
 s(t)&=-\frac{1+t}{2}\log\frac{1+t}{2}
       -\frac{1-t}{2}\log\frac{1-t}{2},\\
 c_k&=\cos\{\pi/(2k)\},\qquad
 B_c(t_0)=s'(t_0)^2(c^{-2}-t_0^2).
\end{align*}
Write
\[
 R_{n,H}^{c}(t_0,s)
 =\inf_T\sup_{|h|\leq H}
 \E_{t_0+h/\sqrt n,c}
 \left[T-\sqrt n\{s(t_0+h/\sqrt n)-s(t_0)\}\right]^2.
\]
Let \(\mathfrak M_{n,k}(t_0;H)\) be the corresponding worst-context
menu minimax risk and let
\(\mathfrak R_{n,\star}(t_0;H)=R_{n,H}^{1}(t_0,s)\).
Thus the equatorial frontier in the main paper is
\(F_{n,k}(H)=\mathfrak M_{n,k}(t_0;H)
             -\mathfrak R_{n,\star}(t_0;H)\).
The estimator and van Trees ingredients are standard.  We restate the
fixed-menu qubit calculation needed below so that the proof is self-contained.
What must be shown is that the constants remain uniform over the triangular
array of menu sizes and that the menu and oracle risks can be coupled before
subtraction.
Equally spaced projective axes have worst overlap \(c_k\).  Conversely, every
set of at most \(k\) ambient projective axes has an equatorial context whose
largest absolute overlap is at most \(c_k\): discard axes with zero equatorial
projection, project the others to the circle, and take a midpoint of a largest
projective gap.  Since \(c_k\geq c_2=1/\sqrt2\), all constants below are
uniform in \(k\).

For a fixed overlap \(c\in[c_2,1]\), write
\[
 J\sim\operatorname{Bin}\{n,(1+ct)/2\},\qquad
 t=t_0+h/\sqrt n,
\]
and let
\[
 \widetilde t=(2J/n-1)/c,\qquad
 \widehat t=\operatorname{proj}_{[\tau/2,\,1-\tau/2]}(\widetilde t).
\]
The estimator
\(\sqrt n\{s(\widehat t)-s(t_0)\}\) is admissible for the local target.
Uniformly over \(|h|\leq H\), exact centered binomial moments and Taylor's
theorem on the clipping interval give
\begin{equation}
 \sup_{|h|\leq H}
 \E\left[
  \sqrt n\{s(\widehat t)-s(t)\}
 \right]^2
 \leq B_c(t_0)
 +C_\tau\{H/\sqrt n+n^{-1}+ne^{-\gamma_\tau n}\}.
 \label{eq:s-baseline-upper}
\end{equation}
Indeed, before replacing \(t\) by \(t_0\), the leading term is
\[
 s'(t)^2(c^{-2}-t^2),
\]
and the clipping probability is at most
\(2\exp(-nc^2\tau^2/32)\).

For the lower bound, use on \([-H,H]\) the rescaled cosine-squared prior.
Its prior Fisher information is \(\pi^2/H^2\).  The van Trees inequality
\citep{GillLevit1995},
uniform continuity of \(s'\), and the binary information
\(I_c(t)=c^2/(1-c^2t^2)\) yield
\begin{equation}
 R_{n,H}^{c}(t_0,s)
 \geq B_c(t_0)-C_\tau\{H/\sqrt n+H^{-2}\}.
 \label{eq:s-baseline-lower}
\end{equation}
Both displays hold at \(c=1\) as well.

Set \(H=n^{1/6}\).  For the equally spaced menu,
\eqref{eq:s-baseline-upper} and monotonicity in the overlap give
\[
 \mathfrak M_{n,k}(t_0;H)
 \leq B_{c_k}(t_0)+C_\tau n^{-1/3}.
\]
For an arbitrary menu, choose the largest-gap context described above.
Its experiment is a garbling of the coefficient-\(c_k\) experiment, so
Blackwell monotonicity and \eqref{eq:s-baseline-lower} give the reverse
inequality with \(-C_\tau n^{-1/3}\).  The same upper and lower calculations
at \(c=1\) give the oracle risk
\(B_1(t_0)+O_\tau(n^{-1/3})\).  Subtraction and
\[
 B_{c_k}(t_0)-B_1(t_0)
 =s'(t_0)^2(c_k^{-2}-1)
 =\operatorname{artanh}^2(t_0)\tan^2\{\pi/(2k)\}
\]
prove the baseline theorem.  Finally,
\(\tan^2\{\pi/(2k)\}\asymp k^{-2}\), so the absolute remainder is
little-oh of the frontier when \(k=o(n^{1/6})\).

\bibliographystyle{imsart-nameyear}
\bibliography{refs}